\title{Stable Systems with Power Law Conditions for Poisson Hail}
\author{Thomas Mountford \thanks{
		\'{E}cole Polytechnique F\'{e}d\'{e}rale de Lausanne
	}        
\and 
Zhe Wang\thanks{
		\'{E}cole Polytechnique F\'{e}d\'{e}rale de Lausanne
	}        
}
\documentclass[11pt,a4paper]{article}

\usepackage[english]{babel}
\usepackage[margin=1.0 in]{geometry}
\usepackage[latin1]{inputenc}

\usepackage{amsmath}
\usepackage{amsfonts}
\usepackage{amssymb}
\usepackage{fourier}
\usepackage{graphicx}
\usepackage{tikz}
\usepackage[inline]{enumitem}
\usepackage{hyperref}

\newtheorem{theorem}{Theorem}[section]
\newtheorem{corollary}{Corollary}[section]
\newtheorem{lemma}{Lemma}[section]
\newtheorem{proposition}{Proposition}[section]
\newtheorem{remark}{Remark}

\numberwithin{equation}{section}

\newenvironment{proof}{{\sc Proof}:}{~\hfill $\square$}

\def\abs#1{\left\vert #1 \right\vert}

\newcommand{\rcom}[1]{\textcolor{black}{#1}}
\newcommand{\rcomm}[1]{\textcolor{black}{#1}}
\newcommand{\edt}[1]{\textcolor{black}{#1}} %12.02
\newcommand{\edtt}[1]{\textcolor{black}{#1}} %12.14
\newcommand{\edttt}[1]{\textcolor{black}{#1}} %12.20
\newcommand{\edtf}[1]{\textcolor{black}{#1}} %01.09
\newcommand{\edtff}[1]{\textcolor{black}{#1}} %01.15
\newcommand{\edtfff}[1]{\textcolor{black}{#1}} %01.30
\newcommand{\edttff}[1]{\textcolor{black}{#1}} %02.09.21
\newcommand{\edttfff}[1]{\textcolor{black}{#1}} %05.07.21
\newcommand{\edts}[1]{\textcolor{black}{#1}} %10.07.21
\newcommand{\edtts}[1]{\textcolor{black}{#1}} %12.20.21
\newcommand{\edtss}[1]{\textcolor{black}{#1}}  %01.15.22
\newcommand{\edtsss}[1]{\textcolor{black}{#1}}%03.09.22
\newcommand{\ledd}[1]{\textcolor{black}{#1}}

\newcommand{\G}{\mathcal{G}}

\begin{document}
	\newpage
	\maketitle
	\begin{abstract} 
We consider Poisson hail models and characterize up to boundaries the collection of critical moments which {guarantee stability.} In particular, we treat the case of infinite speed of propagation.
	\end{abstract}
	
	\section{Introduction}
\subsection{Model and dynamics}
%\edtt{The Poisson hail problem is a space-time stochastic growth model introduced by \cite{11BF}}.
%In this article, we extend the result obtained by Foss, Konstantopoulos and Mountford \cite{18FKM} in any dimension $d\geq 1$

In this article we treat the Poisson hail model introduced in \cite{11BF}. This can be viewed as a system of interacting queues.  In this system $\mathbb{Z}^d$  represents a collection of servers.  As with a classical system, jobs arrive at random times requiring a random amount of service.  The ``interacting" part is that in the
Poisson hail system the arriving job requires  simultaneous service from a random {\it team} of  servers, 
that is a finite subset of $\mathbb{Z}^d$.   Each individual server operates a first come first served (FCFS)
policy, which in this context means that server {$x \in  \mathbb{Z}^d$} only works on a job once all
previously arrived jobs that required service from $x$ have been completed. Once a team has begun a job they work continuously on the job until completion. The FCFS condition implies, for example, that $x$ may be unable to work on any job at a given moment due to the prior arrival of a job requiring a large team of servers (among them $x$) on which no work is possible until a yet earlier job requiring the work of some server $y$ (also in this big team) is completed. We now define our model more precisely. 

	Let $\Phi=(\Phi_x)_{x \in \mathbb{Z}^d}$ be a collection of independent identically distributed marked  Poisson processes on $\mathbb{R}_+$ with marks in $\mathbb{R}_+\times 2^{\mathbb{Z}^d}$ where the second coordinate  is stipulated to be a finite subset of $\mathbb{Z}^d$.  Each point in $\Phi_x$ is denoted by $(t,\tau,B)$, and the triple corresponds to a job ``centered" at $x$. 
The $t$ stands for the arrival time of the job, while the pair $(\tau,B)$ stands for the service time $\tau$ and the team of servers $x+B$ required for the job.   Here ``centered" has no geometrical meaning beyond requiring $B$ to contain the origin $\mathbf{0}$.
The collection of points $N_x= \{t: \exists (t,\tau,B) \in \Phi_x \}$ is a Poisson process with rate $ \lambda$
for each $x$, and the pairs $(\tau,B)$ are assumed to be i.i.d.  over jobs arriving. Thus in our article, the model is translation invariant.
 We make the simplifying assumption that $B$ is a.s. always a cube, as in \cite{11BF}, with a center at the origin and a radius $R$ (in $l_\infty$-norm, $\abs{x} = \max_{i=1}^d \abs{x_i}$).  As we are to give a sufficient condition for stability (which is to be defined in subsection \ref{subsec: Stability and Results} below),  this is not a very restrictive assumption.
We also discuss general shapes for jobs at the end of subsection \ref{subsec: Stability and Results}, see Remark \ref{remark: estimates for further stable model}.   
Thus a job arriving at server $x$ at time $t$ can be  thought of as a couple $(\tau,R)$, rather than a couple $(\tau,B)$.
$(\tau,R)$ are the {\it sizes} of the job{:} $\tau$ the temporal size and $R$ the spatial. As stated
these ``marks" are i.i.d.  over arrival points and servers.  When we write $P((\tau,R) \in A)$ or $P(R \in B)$ we are referring to the underlying probability distribution for the marks. We write $\tilde P $ for the law of couple $(\tau,R)$ ($( \tau, R) \sim \tilde P$)
so $\tilde P (A) = P((\tau,R)\in A)$. Throughout the paper we assume that the system is nontrivial in the sense that there is interaction between the servers, that is $\tilde P (\mathbb{R}_+ \times\{0\}) < 1$.  %We also suppose that
%$\tilde P ((0,\infty ) \times \mathbb{Z}_+) =1$.
%A point $(t,\tau,B) $ in $\Phi_x$,  signifies that at time $t$ a job arrived needing $\tau$ units of service from the team $x+B$. 
 $(t,\tau,R) $ in $\Phi_x$ signifies that at time $t$ a job arrived needing $\tau$ units of service from the team $x+[-R,R]^d$
(or a job $(\tau,R)$ arrived at time $t$). It should not be clear that such a system defines a queuing process since in general it does not. 

We now describe how a queuing system that respects the FCFC stipulation and corresponds to the given job arrivals may be constructed.
We do not claim that this is the only means of constructing a queueing system that corresponds to job arrivals $\{\Phi_x\}_{x \in \mathbb{Z}^d}$.
We first need to assume that for a single site $x$ (and therefore for all sites by translation invariance) the arrival rate of jobs requiring service from $x$ is finite, i.e.,
\begin{equation} \label{finiterate}
\lambda \sum_{y \in \mathbb{Z}^d} P( \mathbf{0} \in B+y) < \infty . \end{equation}
Under the assumption of nontriviality, this rate is strictly above $\lambda $.
We take a non random sequence of finite subsets of $\mathbb{Z}^d$  that increase up to $\mathbb{Z}^d, \ \varXi_n, n=1,2 \ldots$.
We have by our finite rate assumption that on time interval $[0,n]$, there are only a finite number, $N_n$, of jobs arriving
that require service from some server $x \in \varXi_n$. For fixed $n$, the arrival times are a.s. distinct and ordered
\[
0 < t^n_1 < t^n_2 < \ldots < t^n_{N_n} < n.
\]
The $t^n_i $ and $N_n$ correspond to $\varXi_n $ and not to a particular $x $ in it. We construct a (stage $n$) FCFS queueing system based on this finite set of jobs in a straightforward way:  the job $(\tau^n_r, R^n_r)$ arriving at $x^n_r $ at time  $t^n_r $ will be served on time interval $[\sigma^n_r, \sigma^n_r +\tau^n_r)$ by the serving team $x^n_r + [-R^n_r,R^n_r]^d$ where the
$\sigma^n_r$ are recursively defined (for $n$ fixed) by $\sigma^n_1 =  t^n_1$ and for $1<r\leq N^n$,
\[
\sigma^n_r =\max\bigg\{\max _{x \in  x^n_r + [-R^n_r,R^n_r]^d}   \left(  \max _{j<r: x \in x^n_j + [-R^n_j,R^n_j]^d }
\sigma^n_j + \tau^n_j  \right) , t^n_r \bigg\}.
\]
\noindent
For a fixed job $(\tau,R)$ arriving at a $(t,x)$, then when $t \leq n $ and $x \in \varXi_n$, $t$ is on the above list, i.e.,  $t = t^n_{m(n)} $ for some $1 \leq m(n) \leq N_n $ for $n$ large. 
 As $n$ increases $m(n) $ and $\sigma^n_{m(n)} $ 
increase with $n$ once $m(n)$ is well defined. If for each $x \in \mathbb{Z}^d $ and each job $(t,\tau,R)$ that arrives at  $x$ there exists
finite 
$n_0$ so that 
\begin{equation*}\label{n_0, sigma}
 \forall n \geq n_0, \sigma^n_{m(n)} =  \sigma^{n_0}_{m(n_0)} = \sigma,
\end{equation*}
 then we can define the
interacting queuing system where the job is served by the team $x+[-R,R]^d$ on time interval
$[\sigma, \sigma + \tau)$. 
Suppose that  jobs $(s_1, \tau_1, R_1)$ and  $(s_2, \tau_2, R_2)$ arrive at $x_1 $ and $x_2 $ respectively, with $s_1 < s_2$ and that  $(x_1 +[-R_1,R_1]^d )\cap ( x_2 +[-R_2,R_2]^d )  \ne \emptyset $.
Let the beginning service times for the two jobs at stage $n$ (again for all $n$ sufficiently large)  be respectively $\sigma^n_{m^1(n)} $ and $\sigma^n_{m^2(n)} $.
Then we have from the FCFS policy applied to each stage $n$ finite queueing system $\sigma^n_{m^2(n)} \geq \sigma^n_{m^1(n)}+ \tau_1 $.
Thus the final queueing system respects FCFS. The choice of $(\varXi_n)_n$ above can be arbitrary. For any two sequences $(\varXi_n)_n$ and $(\tilde{\varXi}_n)_n$, there exists $\bar{n}$ and $\underline{n}$ such that $\varXi_{\underline{n}} \subset \tilde{\varXi}_n \subset \varXi_{\bar{n}},$ and therefore, we get from monotonicity that $\sigma_{m(\underline{n})}^{\underline{n}} \leq \tilde{\sigma}_{m({n})}^{n} \leq\sigma_{m(\bar{n})}^{\bar{n}}.$ 
%be respectively $\sigma^n_{m^1(n)} $ and $\sigma^n_{m^1(n)} $and will not depend on the particular choice of the sets $\{A_n\}_{n \geq 1}$.   
  
  To see when a finite $\sigma$ might exist, we fix $x$ and $t$ and consider a dual model $\{\G^{x,t}_s\}_{0 \leq s \leq t} \ = \ \{\G_s\}_{0 \leq s \leq t}$.  This is defined by the rules
\begin{enumerate}
\item [i]
$\G_0 = \G^ \prime _0 = \{x\}$, $\tau_0 = 0$ (here $ (\G^ \prime _i)_{i \geq 0}$ is the jump chain for $\G _.$)
\item [ii]
for $i \geq 1, \  T_i \ = \ \inf \{s> T_{i-1}: $ so that a job arrives at $t-s$ requiring service from some \ledd{$ y \in \G^ \prime _{i-1} \}$}. Let $x_i + [-R_i, R_i ] ^d$ be the service team for this job. $ \G^ \prime _{i} =  \G^ \prime _{i-1} \cup \left(x_i + [-R_i, R_i ] ^d\right)$.
\item [iii]
for $s \in [T_i , T_{i+1} ) \ \G_{s} =\ G^ \prime _{i} $;  for $T_\infty = \lim_i T_i, \G_s = \mathbb{Z}^d $ on $[t \wedge T_\infty ,t]$.
%$(s,\tau,B)$ arrives at $u$ then $\G_s = u+B \cup \G_{s_-}$ ($\G_{s_-}$ being the lefthand limit at $s$ which exists by (ii))
\end{enumerate}
\noindent
If $T_\infty \leq t$, we say the dual explodes.
This dual model is a similar object to the duals of interacting particle systems.  %Under the finite rate assumption,this will exist as long as $\vert \G_s  \vert$ does not explode. 
 It is easily seen from \eqref{finiterate} that for all $x,t$, if $\G^{x,t}$ does not explode then for every job arriving requiring service from $x$ in time interval $[0,t]$ there exists $\sigma$  so that for all $n$ large $\sigma^n_{m(n)}  = \sigma$ and so the queuing system is defined if $\G^{x,t}$ does not explode for every $x$ and $t$. In the subsection \ref{subsec: admissible path} below, we will see the connection between the dual $\G^{x,t}$ and admissible paths.

 A central quantity of the model is the workload at site $x$ and time $t$, and we denote it by $W(t,x)$. We discuss the workload of a system heuristically and will provide a formal definition in subsection \ref{subsec: admissible path}. Intuitively speaking, $W(t,x)$ is the additional time required {beyond} time $t$ for server $x$ to have serviced all jobs (requiring service from $x$) arriving on time interval $[0,t]$.% if no new jobs arrive \edtff{involving the} server after time $t$.
  We can understand the dynamics of $W(t,x)$ with the following two equations. Suppose  job $(\tau,R )$ arrives at server $y$ at time $t$.% and it has a sizes $\tau$ and $R$. 
 Then by the FCFS rule,
\begin{equation} \label{eq: updating rule}
W(t+,x)= \begin{cases}
\sup_{ \abs{z-y} \leq R } W(t-,z) + \tau,  \text{ if } \abs{x-y} \leq R, \\
W(t_-,x), \text{ otherwise,}
\end{cases}
\end{equation} where $\abs{x-y}$ stands for the $l_\infty$-norm in $\mathbb{Z}^d$.  By convention, we always assume that \ledd{$t \rightarrow W(t,x)$} is right continuous: $ W(t,x) =  W(t+,x)$. Suppose no job {requiring service from} site $x$ arrives during time $[s,t]$. Then $W(t,x)$ decreases linearly in $t$ at rate $1$ \edtff{until zero:}
\begin{equation} \label{eq: updating rule 2}
W(t,x) = \max \{ \left(W(s,x) - (t-s)\right),0\},
\end{equation}
see \cite{16BCF,11BF,18FKM}.
The following initial condition for $W(t,x)$ is in force throughout the paper: for all $x$ in $\mathbb{Z}^d$
\begin{equation}\label{eq: initial condition}
W(0,x)=0.
\end{equation}

	One way to visualize the model and the workload $W(t,x)$ is to think of jobs as hailstones falling randomly on a hot ground, and this model will be similar to the well-known game, Tetris! A job $(\tau,R)$ received by a server $x$ can be viewed as a hailstone with a base \ledd{$x+[-R,R]^d$} and a height $\tau$. The hailstone falls on sites $x+[-R,R]^d$, and the FCFS rule requires the hailstone to fall on top of all previously arrived hailstones that required service from a server in  $x+[-R,R]^d$. The hailstone starts to melt at a constant rate 1 once its base touches the hot ground. In this way, $W(t,x)$ can be interpreted as the height function at site $x$ evolving in time $t$. \edtsss{See Figure \ref{fig: tetris} for an illustration of how workload changes when a job arrives. In this example, we fix a time $t$, and a job with size $(\tau,R) = (1,2)$ arrives at site $-1$. Before its arrival, the workload at sites $-4,-3,2$ is $2$, at sites $-2,-1, 0, 3$ is $1$, and at sites $1, 4$ is $3$. After its arrival, the workload at sites $-4, 2$ is $2$, at sites $-3,-2,-1,$ is $4$, at site $3$ is $1$, and at site $4$ is $3$.}

	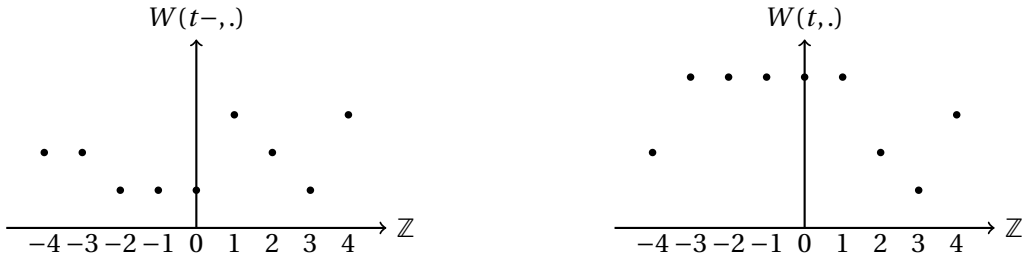
\begin{figure}[h!]
		\begin{center}
			\begin{tikzpicture}[thick]
				\draw[->]
				(-5,0) -- (-5,2.5) node[above]{$W(t-,.)$} ;
				\draw[->]
				(-7.5,0) -- (-2.5,0) node[right]{$\mathbb{Z}$} ;
			%draw intergers
%					\draw
%					(-7.5,-0.2) node{$-5$};
						\draw
						(-7,-0.2) node{$-4$};
						\draw
						(-6.5,-0.2) node{$-3$};
						\draw
						(-6,-0.2) node{$-2$};
						\draw
						(-5.5,-0.2) node{$-1$};
						\draw
						(-5,-0.2) node{$0$};
						\draw
						(-4.5,-0.2) node{$1$};
						\draw
						(-4,-0.2) node{$2$};
						\draw
						(-3.5,-0.2) node{$3$};
						\draw
						(-3,-0.2) node{$4$};
%						\draw
%						(-2.5,-0.2) node{$5$};
						
		%draw workload before
%	\filldraw[black] (-7.5,1) circle (1pt);
	\filldraw[black] (-7,1) circle (1pt);
	\filldraw[black] (-6.5,1) circle (1pt);
	\filldraw[black] (-6,.5) circle (1pt);
	\filldraw[black] (-5.5,.5) circle (1pt);
	\filldraw[black] (-5,.5) circle (1pt);
	\filldraw[black] (-4.5,1.5) circle (1pt);
	\filldraw[black] (-4,1) circle (1pt);
	\filldraw[black] (-3.5,.5) circle (1pt);
	\filldraw[black] (-3,1.5) circle (1pt);

	% after job arrivals			
				\draw[->]
				(3,0) -- (3,2.5) node[above]{$W(t,.)$} ;
				\draw[->]
				(0.5,0) -- (5.5,0) node[right]{$\mathbb{Z}$} ;
%				\draw
%				(0.5,-0.2) node{$-5$};
				\draw
				(1,-0.2) node{$-4$};
				\draw
				(1.5,-0.2) node{$-3$};
				\draw
				(2,-0.2) node{$-2$};
				\draw
				(2.5,-0.2) node{$-1$};
				\draw
				(3,-0.2) node{$0$};
				\draw
				(3.5,-0.2) node{$1$};
				\draw
				(4,-0.2) node{$2$};
				\draw
				(4.5,-0.2) node{$3$};
				\draw
				(5,-0.2) node{$4$};
%				\draw
%				(5.5,-0.2) node{$5$};
			
				%draw workload after
			%	\filldraw[black] (-7.5,1) circle (1pt)
			\filldraw[black] (1,1) circle (1pt);
			\filldraw[black] (1.5,2) circle (1pt);
			\filldraw[black] (2,2) circle (1pt);
			\filldraw[black] (2.5,2) circle (1pt);
			\filldraw[black] (3,2) circle (1pt);
			\filldraw[black] (3.5,2) circle (1pt);
			\filldraw[black] (4,1) circle (1pt);
			\filldraw[black] (4.5,0.5) circle (1pt);
			\filldraw[black] (5,1.5) circle (1pt);
				
			\end{tikzpicture}
			\caption{Workloads before and after a Job Arrival.}\label{fig: tetris}
		\end{center}
	\end{figure}

\subsection{Admissible paths, Workload, and Time Scales}\label{subsec: admissible path}
%The key object of this paper is the workload of a system $\{ \Phi(x)\}_{x\in\mathbb{Z}^d}$. 
 We now rigorously define {workload} via \ledd{\textit{admissible}} paths.
Suppose the arrival times $\left(\rcom{N_x}\right)_{x\in\mathbb{Z}^d}$ and spatial sizes $R$ of jobs are known. For any $0\leq u\leq t$, an admissible path is a piece-wise constant, right-continuous (c\`{a}dl\`{a}g) function \ledd{$\gamma_{}: [u,t] \rightarrow \mathbb{Z}^d$} such that,  
if $\gamma(s) \neq  \gamma(s-) $, there is a job arriving at time $s \in [u,t]$ with center $x$ \edtt{for some $x\in\mathbb{Z}^d$} (equivalently, ${N_x(s)=N_x(s-)+1}$) and sizes $(\tau,R)$, and it {\it intersects} the path $\ledd{\gamma_{}}$ in the sense:
\begin{equation}\label{eq: intersection}
\ledd{\gamma_{}}(s),\ledd{\gamma_{}}(s-) \in B(x,R):=\{y:\abs{y-x}%_{\edts{\infty}}
	\leq R\}.
\end{equation} \ledd{We also use $\gamma_{u,t}$ when we want to emphasize that the admissible path is on a fixed interval $[u,t]$.}
\ledd{So,} given $x \in \mathbb{Z}^d$ and $t >0$, the dual model $\{ \G^{x,t}_s\}_{0\leq s \leq t}$ has
$y \in \G^{x,t}_s$ if and only if there exists an admissible path {$\gamma_{t-s,t}:[t-s,t] \rightarrow \mathbb{Z}^d $}
with $\gamma _{t-s,t}(t-s) = y $ and  $\gamma_{t-s,t} (t) = x $.
For each admissible path $\gamma_{u,t}$, we can define its {\it load } 
\begin{equation} \label{eq: workload of a path}
	U(\gamma_{u,t})=\sum \tau_i,
\end{equation} by summing over all jobs intersecting the admissible path $\gamma_{u,t}$ in the sense of \eqref{eq: intersection}, and assign it a score 
\begin{equation} \label{eq: path workload and score}
V(\gamma_{u,t}) = U(\gamma_{u,t})-(t-u)=\sum \edtt{\tau_i} - (t-u).
\end{equation} 
The workload at the site $(t,x)$ is the maximal score over all admissible {paths} \ledd{$\gamma_{u,t}$} starting at some positive time \ledd{$u\in [0,t]$} and ending at $\ledd{\gamma_{u,t}}(t)=x$, see \cite{11BF},
\begin{equation} \label{eq: def of workload}
W(t,x) = \sup_{0\leq u\leq t} \left( \sup_{\substack{\gamma_{u,t}(t)=x,\\ \text{$\gamma_{u,t}$ admissible }}}V(\gamma_{u,t}) \right).  
\end{equation} 
If one admits the value $\infty $ as a possible value, then $W(t,x)$ is well defined given any $\{ \Phi(x)\}_{x\in\mathbb{Z}^d}$, {whether a queuing model can be defined or not}.  However in the case where the dual model does not explode we have that \ledd{the random variable $W(t,x)$ equals the additional time that the queuing model required until $x$ has served all jobs that arrived before the time $t$, and required service from $x$}, see \cite{16BCF,11BF,18FKM}.
If one takes \eqref{eq: def of workload} as a definition of $W(t,x)$ without referencing to a queuing system, then in fact, {(see Remark \ref{remark: estimates for further stable model} after Theorem \ref{thm: unstable models in II, III, IV} )} there are nontrivial models that have workload $W(t,x) =\infty$ almost surely for any time $t>0$. However, this is not the focus of our article. Instead, we study the \textit{stability} of the system, and derive sufficient conditions for it, though these sufficient conditions for stability imply that the workload is finite for any positive time $t$ almost surely. From \eqref{eq: def of workload}, one can verify the properties \eqref{eq: updating rule} and \eqref{eq: updating rule 2}. See Figure \ref{fig: illustration, admissible path, score,workload} for an illustration for an admissible path, its score and workload. In this example, red horizontal intervals represent job arrivals with temporal sizes $\tau_i$. The blue line represents an admissible path $\gamma_{u,t}$, which passes jobs $\tau_2,\tau_4,\tau_6, \tau_7 $. The  load of $\gamma_{u,t}$ is $U(\gamma_{u,t})= \tau_2 +\tau_4+\tau_6+ \tau_7$, and $\gamma_{u,t}$ has a score $V(\gamma_{u,t})= \tau_2 +\tau_4+\tau_6+ \tau_7 - (t-u)$. $W(t,\mathbf{0})$ is the supremum of these scores over all such paths starting
	from some $(u, y)$ and ending at $(t, \mathbf{0})$.
	
\begin{figure}[h!]
	\begin{center}
	\begin{tikzpicture}[thick]
		\draw[->]
		(-5.5,-2) -- (-5.5,4) node[left]{$time$} ;
		\draw[->]
		(-6,0) -- (3,0) node[right]{$\mathbb{Z}$} ;
		
		\draw[dashed]
		(-6,3.75) -- (3,3.75) node[right]{$t$}  ;
		\draw[dashed]
		(-6,-.3) -- (3,-.3) node[right]{$u$};
		\draw[dashed]
		(-6,-1.8) -- (3,-1.8) node[right]{$0$};

		%draw intergers
		\foreach \x in {-4,...,3}
			\draw (-1+\x,4) node[above]{\x}-- (-1+\x ,-1.8);
			
		%draw jobs
		\draw[red,very thick] (-1+-4, -0.6)--(-1+-3, -0.6)  node[below]{$\tau_1$} --(-1+-2, -0.6);
		\draw[red,very thick] (-1+-3, 0.2)--(-1+-1, 0.2)  node[below]{$\tau_2$} --(-1+1, 0.2);
		\draw[red,very thick] (-1+-0.5, 0.8)--(-1+0, 0.8)  node[below]{$\tau_3$} --(-1+0.5, 0.8);
			\draw[red,very thick] (-1+-1.5, 1.2)--(-1+-1, 1.2)  node[below]{$\tau_4$} --(-1+-0.5, 1.2);
				\draw[red,very thick] (-1+-0.5, 0.8)--(-1+0, 0.8)  node[below]{$\tau_5$} --(-1+0.5, 0.8);
		\draw[red,very thick] (-1+-2.8, 1.7)--(-1+0, 1.7)  node[below]{$\tau_6$} --(-1+3.2, 1.7);
			\draw[red,very thick] (-1+-0.2, 3)--(-1+2, 3)  node[below]{$\tau_7$} --(-1+4, 3);
			\draw[red,very thick] (-1+-4, 3.3)--(-1+-3, 3.3)  node[below]{$\tau_8$} --(-1+-2,3.3);
			
	% draw an admissible path
	\draw[blue,ultra thick] (-1+-2, -.3)--(-1+-2, .2)
	--(-1+-1, .2)--(-1+-1, 1.7) -- (-1+3, 1.7)
	-- (-1+3, 2.2) node[right]{$\gamma_{u,t}$} 
	-- (-1+3, 3) --  (-1+0, 3)--  (-1+0, 3.75);
	
	\end{tikzpicture}
		\caption{Admissible Path, Score, and Workload.}\label{fig: illustration, admissible path, score,workload}
	\end{center}
\end{figure}
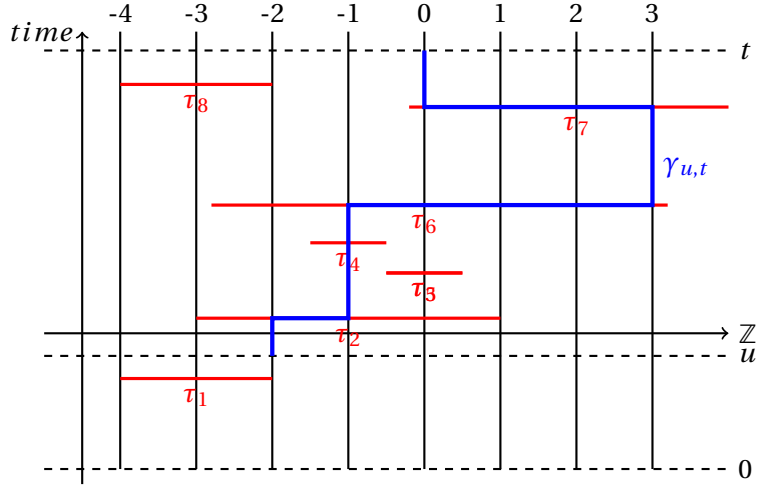
%monotonicity properties in Remark \ref{remark: estimates for further stable model} with standard arguments.

We set the initial starting point $\gamma(0)=\edtt{\mathbf{0}}$ to obtain a new random variable 
\begin{equation} \label{eq: alternative def}
\tilde{W}(t,\edtt{\mathbf{0}}) =\sup_{ 0\leq u\leq t} \sup_{\substack{\gamma_{0,u}(0)=\mathbf{0},\\ \text{$\gamma_{0,u}$ admissible}}} V(\gamma_{0,u}),
\end{equation} 
which by time reversibility and homogeneity of Poisson processes has the same distribution as $W(t,x)$ for any $\ledd{(t,x)}$. An advantage of using $\tilde{W}(t,\edtt{\mathbf{0}})$ is that $\tilde{W}(t,\edtt{\mathbf{0}})$ is increasing in $t$ almost surely. This monotonicity is different from those in {Remark \ref{remark: estimates for further stable model}, below} and we apply it in Lemma \ref{lm: BC lemma}. Since the growth estimate in Lemma \ref{lm: BC lemma} is related to $\sup U(\gamma_{0,t})$, we also remark that $\tilde{W}(t,\edtt{\mathbf{0}})$ involves all admissible paths with the initial point $(0,\edtt{\mathbf{0}})$  and some ending time $u\leq t$, whereas $W(t, \mathbf{0}) $ 
involves admissible paths with the fixed final point $(t,\mathbf{0})$.  Another advantage of $\tilde{W}(t,\edtt{\mathbf{0}}) $ is that it enjoys {superadditivity properties.}

\subsection{Stability and Results} \label{subsec: Stability and Results} 
We now define {\it stability}: 
\ledd{From \eqref{eq: def of workload}, given a system $\{ \Phi(x)\}_{x\in\mathbb{Z}^d}$, we see that $W(t,\mathbf{0})$ is stochastically increasing in $t$}.  Given a law $\tilde P$ for the couple $(\tau, R)$, we say that the family of $\{ \Phi(x)\}_{x\in\mathbb{Z}^d}$ as $\lambda $ varies, is {\it stable} if there exists a $\lambda_1>0$ so that for rate $0 < \lambda < \lambda_1$, the system  $\{ \Phi(x)\}_{x\in\mathbb{Z}^d}$ is such that $\{W(t,\mathbf{0})\}_{t \geq 0}$ is tight.  
{In this case,  we also say $\tilde P$ is tight.}
We are interested in the stability of systems. We stress that the notion of stability depends on the law $\tilde P$ and not on the existence of a queueing model.
It is not clear that a queueing system can be defined when our dual model explodes.  But for such a model to exist {meaningfully } it must have its workload given by \eqref{eq: def of workload}.
By translation invariance $\tilde P $ is stable if and only if for any $x \in \mathbb{Z}^d$, $\{W(t,x)\}_{t \geq 0}$ is tight.
The central question in this article is to understand the stability of the system.
It is natural to expect that if $\lambda$ is small, $W(t,\edtt{\mathbf{0}})$ does not grow fast in time, and it decreases due to no arrivals of job but if a job of large spatial size follows a job of large temporal size, the system has many sites with large workload. 
In fact, one can show one of the three scenarios holds for a system:
\begin{enumerate}[label= \alph*)]
 \item $W(t,\mathbf{0})$ is \textit{infinite} for \edtff{some} time $t>0$ with probability one,
\item the system is unstable but with $W(t,\mathbf{0})$ \textit{finite}  for any time $t>0$ with probability one,% \edtfff{$W(t,\edtt{\mathbf{0}})$ \textit{finite}} for any time $t>0$ with probability one,
\item or the system is stable with $W(t,\mathbf{0})$ \textit{finite} for any time $t>0$ with probability one.
\end{enumerate}

In \cite{18FKM}, 
Foss, Konstantopoulos, and Mountford showed that a finite $d+1+\epsilon$-th moment for the sum of $R$ and $\tau$ is sufficient for stability, and they also showed that for every $0<\epsilon<d+1$, there is an unstable system with a finite $d+1-\epsilon$-th moment for the sum of $R$ and $\tau$. As a consequence, $d+1$ is a critical moment for the sum of $R$ and $\tau$. Their method is via techniques introduced in the  study of greedy lattice animals of $\mathbb{Z}^d$, see \cite{93CGGK,02M} for instance. We can interpret these two results by considering critical moments for positive random variables $R$ and $\tau$. Let 

\begin{equation}\label{eq: critical moment}
\alpha := \sup\{a\geq -d : \mathbb{E}\left[R^{d+a}\right]<\infty \},\quad  \beta := \sup\{b\geq -1: \mathbb{E}\left[ \tau ^{1+b}\right]<\infty \}.
\end{equation}
So every law $\tilde P $ defines a point $(\alpha, \beta )$ in$ [-d,\infty] \times  [-1,\infty].$ We say that a law $\tilde P $ is in a region $A$ if the corresponding pair $(\alpha,\beta)$ is in $A$. We divide up the space for $(\alpha,\beta)$ into regions, see Figure \ref{fig: phase diagram}.
\begin{figure}[h!]
	\begin{center}
		\begin{tikzpicture}[thick]
		\draw[->]
		(-0.2,0) -- (-0.2,3) node[above]{$\alpha$} ;
		\draw[->]
		(-0.2,0) -- (6,0) node[right]{$\beta$} ;
		\draw[very thick,line cap=round,smooth,domain=2:6]
		plot[samples=500] (\x, {2/\x}) ;
		\draw (6,0.5) node[right]{$\alpha \beta = d$};
		\draw [dashed]
		(-0.2,1) node[left]{$1$} -- (6,1) ;%{$\alpha =1$} ;
		\draw [dashed]
		(2,0) node[below]{$d$} -- (2,3); %node[above]{$\beta =d$} ;
		\draw
		(4,2.2) node{$I$, Stable} ;
		\draw
		(0.9,0.4) node{$II$, Unstable};
		\draw
		(0.9,2.2) node{$III$, Unstable};
		\draw
		(3,0.4) node{$IV$};
		\draw
		(5,0.75) node{$V$};
		\end{tikzpicture}
		\caption{Stability and Critical Moments}\label{fig: phase diagram}.
	\end{center}
\end{figure}
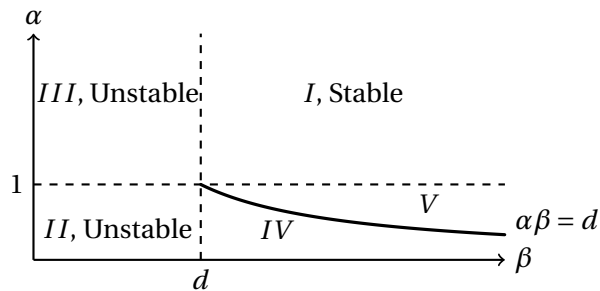
%See Figure \ref{fig: phase diagram}.

The results of  \cite{18FKM} say that the system is always stable in the region I$= \{\alpha>1, \beta>d\}$, and there exists an unstable system for every point in the region II$= \{\alpha<1, \beta<d\}$. Stability in region III$= \{\alpha>1, \beta<d\}$, region IV$= \{0\leq\alpha<1, \beta>d, \alpha\beta<d \}$, and region \edtff{V$= \{0<\alpha<1, \alpha\beta>d\}$} are not clear from their results.
% In fact, we can show that every nontrivial system (i.e. \edtf{$R>0$}) in regions II and III, ${ \{\beta<d \}}$, is unstable; \edtt{any system not in the first quadrant, \edtff{$\{\alpha,\beta\geq0\}$}, is also unstable.} We can find an unstable system for every point in region IV${= \{\alpha \beta <d, \beta >d , \alpha <1\}}.$ We summarize the two results in Theorem \ref{thm: unstable models in II, III, IV}, and we will prove Theorem \ref{thm: unstable models in II, III, IV} in Section \ref{sec:growth}. 
Our results are 

\begin{theorem} \label{thm: unstable models in II, III, IV} 
		Consider the Poisson hail problem in any dimension $d\geq 1$. Let $d+\alpha, 1+\beta$ be the critical moments for $R$ and $\tau$ as defined in  \eqref{eq: critical moment}. Then
	\begin{enumerate}
		\item every nontrivial system in region II and III, $\edtt{ \{\beta<d \}}$ is unstable; 

		\item \edtt{
		every nontrivial system in $\{\alpha<0 \text{ or } \beta \leq 0\}$ is  unstable;}
		
		\item for every point $(\alpha,\beta)$ in region IV ${= \{\alpha \beta <d, \beta >d , 0<\alpha <1\}}$, there is an unstable system with  
		\[  \mathbb{E}\left[R^{d+\alpha} + \tau^{1+\beta} \right]<\infty.     \]
	\end{enumerate}
\end{theorem}

and

\begin{theorem}\label{thm: stability in V}
	Consider the Poisson hail problem in dimension $d\geq1$. Any system with parameters $(\alpha,\beta)$ in region {V$= \{\alpha \beta >d, \edtff{0<\alpha <1} \}$} or \edtf{on the ray $\{(\alpha,\beta): \beta> d,\alpha =1 \}$} is stable.
%	 we can find two increasing sequences $(S_i)_{i\geq 1}, (T_i)_{i\geq1}$, a pair $(\beta', \alpha')$, and an arrival process of jobs $(\tau',R')$ such that
%	\begin{align}
%	P(S_i \leq R <S_{i+1}) &\leq P(R'= S_{i+1})) = c_1 S_{i+1}^{-(d+\alpha')}, \quad \rcom{and }\quad  R' \in \{S_i\}_{i\geq 1}, \label{eq: dominance growth1}
%	\\
%	P(T_i\leq \tau <T_{i+1}) &\leq P(\tau'= T_{i+1})) = c_2 T_i^{-(1+\beta')}, \quad \rcom{and }\quad  \tau' \in \{T_i\}_{i\geq 1}, \label{eq: dominance growth2}
%	\end{align}
%	for some normalizing constants $c_1,c_2$, and the workload $W'(t,x)$ in the new system stochastically dominates the workload in the generic system $W(t,x)$,
%	\[ W'(t,x) \geq W(t,x),\quad  a.s. \] In particular, both systems are stable.
\end{theorem} 

These results (and \cite{22W}) permit a classification of the models according to the parameters $(\alpha, \beta)$. We do not consider boundaries between regions though for some boundaries simple monotonicity
considerations permit a classification.
The proofs of these various {classifications are obtained }as follows.
\begin{enumerate}
\item [I]
This is shown in \cite{18FKM}.
\item [II]
This is the content of Theorem \ref{thm: unstable models in II, III, IV},part 1.
\item [III]
This is the content of Theorem \ref{thm: unstable models in II, III, IV},part 1.
\item [IV]
This is the content of Theorem \ref{thm: unstable models in II, III, IV},part 3 and \cite{22W}.
\item [V]
This is the content of Theorem \ref{thm: stability in V}.
\end{enumerate}
\noindent
The region $IV$ is the only one for which the stability or instability of a system with {$(\alpha, \beta)$ in the given region is not determined by the region.  In fact  \cite{22W} shows that for any $(\alpha, \beta) $ }in $IV$, there exist both stable models and unstable models corresponding to $(\alpha, \beta)$.

We end this subsection  remarking on monotonicity properties for the workload $W(t,x)$ which can be found in \cite{18FKM}. One can also verify these properties with { the } formula for workload $W(t,x)$, see \eqref{eq: def of workload} in the {previous} subsection.
\begin{remark}\label{remark: estimates for further stable model}
\edtt{The model has monotonicity properties}: $W(t,x)$ {increases} if we 
		\begin{enumerate*} [label = (\alph*)]
%			\item delay the arrivals of hailstones between time $0$ and $t$, or 
			\item increase the temporal size of the stones, or
			\item enlarge the spatial shape. 
		\end{enumerate*} Due to the monotonicity of the model,
		\begin{enumerate}	
			\item We \ledd{treat} general convex shapes for hailstones in the system. One way is to generalize $R$ as the maximal distances between two points in the hailstone (under \ledd{a} certain norm). By enlarging the spatial shape of the hailstone to a cube with diameter R, we get an upper bound for $W(t,x)$. Together with Theorem \ref{thm: existence model in IV}, we can get a similar upper bound for {workloads} for general shapes. This ``upper bound" may not be optimal if the shape does not have nonempty interior.			
			
			\item We can construct a new system from a generic stable system by the following strategy: for fixed increasing sequences $(S_i), (T_j)$, we enlarge the job sizes $(\tau,R)$ to $(T_{j+1},S_{i+1})$, if $S_i\leq R<S_{i+1}$ and $T_j\leq\tau<T_{j+1}$ for some $i,j$. The new system has sizes $R$ (and $\tau$) with distribution satisfying \eqref{eq: growth1} (and \eqref{eq: temporal dist}, see below), for some parameters $(\alpha,\beta)$ and two normalizing constants ${c_1,c_2>0}$. The parameters $(\alpha,\beta)$ of the new system are typically smaller than $(\alpha_o,\beta_o)$ of the original system, but they can be chosen to be in region V for appropriate $(S_i), (T_j)$. For details, see the { discussion} of Theorem \ref{thm: stability in V} in subsection \ref{subsec: stability in region V}. 
		\end{enumerate}
\end{remark}
	
Systems in regions IV and V behave differently to systems in the region I due to different tail behaviors of $R$. A major difference is that the spatial growth of { admissible paths } can be arbitrarily fast in regions IV \ledd{and} V and the spatial growth introduces different time scales. To { illustrate} this, we consider a system in dimension $d=1$ belonging to region $\{0<\alpha<1\}$, and the spatial size $R$ has distribution of the form 
\begin{align}
		P(R = S_i) & = c_1 S_i^{-(d+\alpha)}, \quad \text{ and }\quad  R \in \{S_i\}_{i\geq 1}, \label{eq: growth1}
\end{align} 
for some increasing sequence $(S_i) $ { increasing at least geometrically fast} and a normalizing constant $c_1$. 
In analyzing $W(t,\mathbf{0})$, it is equivalent to analyze {$\tilde{W}(t,\mathbf{0})$}, which we now do.  In dimension $d=1$, we have the advantage of defining the left-most and right-most points, $L_t$ and $R_t$, reached by admissible paths with initial point $(0,0)$ by time $t$, %{in distribution}
	\[ L_t = \inf_{\substack{\gamma(0)=\mathbf{0},\\ \text{$\gamma$ admissible}}} \gamma(t), \quad R_t = \sup_{\substack{\gamma(0)=\mathbf{0},\\ \text{$\gamma$ admissible}}} \gamma(t).   \]
It is not hard to see that $D(t) = R_t - L_t $ can be compared to two (non independent) compound Poisson processes, 
\begin{equation}\label{eq:stochastic bounds} 
\sum_i \frac{1}{2}S_i\cdot v_i(t)  \leq D(t) \leq \sum_i 2 S_i\cdot V_i(t), \end{equation}
where for $i$ fixed $v_i $ and $V_i$ are not independent but \ledd{$(v_i (t),V_i(t))_{ i\geq 1}$ are independent as $i$ varies, and for each $i$,
$v_i(t)$ and $V_i(t)$ are Poisson variables with rates $\left(\frac{c}{2} S_i^{-\alpha%+\epsilon
	}t\right)_i$ and $\left(2c S_i^{-\alpha%+\epsilon
	}t\right)_i$}. The lower bound indicates that $D(t)$ grows super-linearly,
\[ \lim_{t\to \infty} \frac{D(t)}{t} \geq \lim_{t\to \infty}\frac{v_i(t)}{2t}S_i = \frac{1}{2}S_i^{1-\alpha},  \] while the upper bound indicates $D(s)$ has increment of sizes at most $2S_i$ at time scales $S_i^{\alpha %+ \epsilon
}.$ The different speeds $S_i^{1-\alpha}$ in different time scales $S_i^{\alpha}$ indicate that the {``domain of influence"} $\mathcal{R}_t$ for $\tilde{W}(t,\edtt{\mathbf{0}})$ is growing at different speeds on different time scales too, and this leads to nontrivial growth behaviors of $\tilde{W}(t,\edtt{\mathbf{0}})$. We {use} one example in the proof of Theorem \ref{thm: unstable models in II, III, IV}, where we construct unstable models with $\limsup_{t\to\infty}\frac{W(t,\edtt{\mathbf{0}})}{t} =\infty$ almost surely. In these models, we  find jobs of large temporal sizes with high probability in {time-space} boxes with increasing spatial and temporal sizes. In particular, the ratio of the spatial sizes and the temporal size of these {time-space} boxes diverges as the temporal size grows.

\subsection{Overview of the Proofs and Outline}
The proof of Theorem \ref{thm: unstable models in II, III, IV} is similar to the proof of Theorem 2 in \cite{18FKM}. We {look} into {time-space} boxes of increasing sizes, and consider the probability of obtaining jobs of large temporal sizes in these boxes \ledd{connected} via some admissible path. We can concatenate disjoint boxes so that the numbers of jobs inside are independent and apply %Borel-Cantelli Lemma or 
\edtf{the} \edtt{Law of Large Numbers for independent random variables} to estimate the probability of getting an admissible path with a high score. The detailed proof of Theorem  \ref{thm: unstable models in II, III, IV} is presented in subsection \ref{sec: proof of unstable models}.

 %The proof of Theorem \ref{thm: stability in V} relies on Theorem \ref{thm: existence model in IV} and the monotonicity property previously mentioned. % in Remark \ref{remark: estimates for further stable model}. 

The proof of Theorem  \ref{thm: stability in V} relies on two reductions.
In section 2  we reduce a generic system to a new system by enlarging job sizes to some discrete values as suggested in \ledd{Remark} \ref{remark: estimates for further stable model}. The new system {may} have smaller parameters $(\alpha,\beta)$ though still being in region $V$, and the workload of the new system {dominates} the workload of the original system stochastically. We {choose} the job sizes via two sequences of specific forms, see \eqref{eq: concrete choice of S, T,0}. These values allows us get some conditions, see \eqref{eq: well-separated}--\eqref{eq: useful ones} in section \ref{sec:growth}, under which we {prove} Theorem \ref{thm: existence model in IV}. These conditions are not necessary, but they enable the proof of Theorem \ref{thm: existence model in IV}. 
In section 3, we further reduce the proof Theorem \ref{thm: stability in V} 
to two technical propositions. The principle task essentially reduces to showing Proposition \ref{prop: precise estimate}.

The programme to show Proposition \ref{prop: precise estimate} consists of two steps corresponding to sections 4 and 5.  The objective is to show that for fixed temporal size $T_j$, with large probability as $t$ tends to infinity, all admissible paths \ledd{with a fixed end point} of length $t$ intersect at most $\frac{t}{T_j^{1+ \delta'}}$ such jobs for $\delta'>0$ not depending on $j$.  In Section 4 we make a first step.  We choose an ``appropriate" spatial size $S_{i_0}$
corresponding to $j$ and show through various percolation arguments that with high probability all
admissible paths starting at a given place which do not ``use" job sizes $S_i$ for $i>i_0$ have the desired bounds.  We in fact show that the bounds hold outside exponentially small probabilities.  It is here that the 
condition $\alpha \beta >d$ is used.  In section 5 we use an induction argument on $i$ to show that given an appropriate bound for admissible paths \ledd{that} do not ``use" job sizes $S_k$ for $k>i$, we can find an appropriate bound for admissible paths \ledd{that} do not ``use" job sizes $S_k$ for $k>i+1$.  The price to be paid is that the rate of the exponential bound for the bad events decreases as we pass from $i$ to $i+1$.   Given that we have chosen  the $S_i$ to be \ledd{increasing} superexponentially, this does not pose a problem.
{The paper is concluded in section 6 by putting together the results of sections 4 and 5 to obtain
Proposition \ref{prop: precise estimate}}.

\section{Behaviors around the Critical Curve $\alpha\beta=d$}\label{sec: stability in region V}
In this section, we  first prove Theorem  \ref{thm: unstable models in II, III, IV}, which includes unstable cases in region II,III and IV. We {understand } \edtf{the importance of } the curve $L= \{\alpha\beta = d\}$ from the construction of an unstable model in region IV, see \eqref{eq: critical rates} below. We also look into models with parameters $(\alpha,\beta)$ above the critical curve, $\{\alpha\beta > d\}$. We then show that the proof of Theorem \ref{thm: stability in V} need only be given for $\tilde P$ which are discrete with possible values (for $\tau$ and $R$) well separated.

\subsection{Unstable Models in Regions II, III and Constructing Unstable Models in Region IV}\label{sec: proof of unstable models}
The proof of Theorem \ref{thm: unstable models in II, III, IV} is similar to that in section 5 of \cite{18FKM}. \edtt{It is an application of the graphical construction and %the second Borel-Cantelli Lemma or 
\edtf{the} Law of Large Numbers for independent random variables.} \edt{We investigate any generic model in regions II and III, $\{\beta< d\}$, and also construct an unstable model in region IV$=\{\alpha\beta>d,0\leq \alpha<1\}$.} 

\begin{proof}(Theorem \ref{thm: unstable models in II, III, IV}) %The proof is divided into two cases. For the general case in Region II and III, we choose $\xi=1$, and for the second case when $\alpha \beta <d$, we will construct an unstable model and choose $\xi = {\alpha}$. 
	\begin{enumerate}
		\item \rcomm{(Region II, III)} 
{Let the arrival rate be } $\lambda > 0$.
By the assumption that the system is nontrivial, we can use a large deviation estimate for a continuous time random walk, and get that there exists a strictly positive $c_0= c_0( \lambda )$  and $n_0 < \infty $ such that  {$\forall n \geq n_0$}
%$[-2^k,2^k]^d$ at times $n 2^{k+1}$ and $(2n+1) 2^{k}$, i.e.,
		\[  \inf_ {x\in [-c_0n, c_0 n]^d}   P\left(  \left\{\text{exists an admissible path  $\gamma$ with $\gamma(n) = \mathbf{0}$ and $\gamma(0) =x$} \right\} \right) > 1- e^{-c_0 n} .  \] 

We choose strictly positive $\epsilon < \frac{d - \beta}{3}$.  By the definition of $\beta$ we see that $E(\tau^{1+\beta + \epsilon}) = \ \infty $ so there exists a sequence $n_i$ increasing to infinity such that  
$n_i $ increasing to infinity with
\[
P(\tau \geq 3n_i) \ \geq \ \frac{1}{n_i^{1+\beta + 2 \epsilon}}.
\]
We may assume that $n_i > n_0 \ \forall i \geq 0$.

For any $i \geq 1$, the event that $W(2n_i, \mathbf{0}) \leq n_i$
is contained in the union of two events
\begin{itemize}
\item
$\{ \nexists (t,x) \in  [0,n_i] \times  [-c_0n_i, c_0n_i]^d$ so that a job arrives at $(t,x)$ with $\tau \geq 3n_i \}$.
\item
$\{ \nexists$ an admissible path $\gamma:[n_i, 2 n_i]$ with $\gamma (n_i) = x$ and $\gamma (2n_i) = \mathbf{0}\}$
\end{itemize}

But the probability of the first event is bounded by $e^{- \lambda \frac{(2c_0n_i+1)^d n_i}{n_i^{1+\beta + 2 \epsilon}}}$ which tends to zero as $n_i$ tends to infinity, while the second {conditioned on the existence of an $x$ for the first chosen by nonrandom ordering of sites} has probability bounded by
$e^{-c_0 n_i} $ by our choice of $c_0$.  Thus the distributions {$\{W( 2n_i, \mathbf{0})\}_{i \geq 0}$} are not tight.

	\item (Regions outside \ledd{$\{\alpha \geq 0 ,\beta >0\}$})
{The first case deals with \ledd{$\beta\leq 0 $} so we need only treat $\alpha < 0 $.}
We suppose ${e_1}>0$ fixed so that $P(\tau > {e_1}) >{e_1}.$  We fix $M \geq 4$ and as in the previous case, we fix $\epsilon > 0 $ and less than $\frac{- \alpha }{3}$.  Again we have
the existence of 
$n_i $ increasing to infinity with
\[
P(R\geq 3n_i) \ \geq \ \frac{1}{n_i^{d-\epsilon}}.
\]
For every $i$, the event $\{  W( 1,\mathbf{0})< M{e_1} -1\}$ is contained in the union of
\begin{itemize}
\item
$\{ $ for some $1 \leq k \leq M \ \nexists (x,t) \in [-n_i, n_i]^d \times [1-\frac{k-1}{M}, 1-\frac{2k-1}{2M}]$ so that a job arrives at $(x,t)$ with $R \geq 3n_i \}$.
\item
$\{ \exists 1 \leq k \leq M $ so that  $\nexists \ (x,t) \in [-n_i, n_i]^d \times [1-\frac{2k-1}{2M}, 1-\frac{k}{M}]$ so that  a job arrives at $(x,t)$ with $\tau > {e_1} \}$.
\end{itemize}

So, as above,  $P(W( 1 ,\mathbf{0})< M{e_1} -1)$ is bounded by $Me^{-(2n_i+1)^d \frac{\lambda}{2Mn_i^{d-\epsilon}}} +
Me^{-(2n_i+1)^d \lambda {e_1}/2M} $.  We see the lack of tightness from the arbitrariness of $i$.

		\item (Region IV) For any point $(\alpha,\beta)$ in Region IV$=\{\alpha \beta <d , 0<\alpha <1, \beta >d\}$, {we construct an unstable system having $(\alpha, \beta)$ as its moment parameters.}  Note that unlike the previous two cases, we are not claiming that any distribution {with parameters } in this region is necessarily unstable.

For this system, jobs have discrete spatial and temporal sizes. In particular, for all positive integer $i$, we define $R_i := 2^{(1+\beta)i}$ and $\tau_i := 2^{(d+\alpha) i}$. The jobs have only sizes {$(\tau_i,R_i)$ for some $i\in \mathbb{N}$}, and with probability  
\begin{equation}
		\tilde P (\{(\tau_i, r_i)\})  =  P\left( (\tau, R)  =(  \tau_i ,R_i) \right) = c \cdot 2^{-(d+\alpha)(1+\beta)i} , 
\end{equation}
		where $c$ is a normalizing constant $c = \left(\sum_{i= 1}^{\infty} 2^{-(d+\alpha)(1+\beta)i} \right)^{-1}$. It is easy to check this system is in Region IV, and we only need to show that this system is unstable. The argument is similar to the first case. %We will assume $\lambda <\lambda_d$. 
		We consider \ledd{$W( \frac{\tau_i}{4} ,\mathbf{0})$}. The number, $N_i$, of $(\tau_i,R_i)$
		jobs arriving at a space time point in $ [-\frac{R_i}{4},\frac{R_i}{4}]^d\times[0 ,\frac{\tau_i}{4}] $
		is a Poisson random variable with parameter
		\begin{equation}\label{eq: critical rates}
			c 2^{-(d+1)}R_i^d \tau_i 2^{-(d+\alpha)(1+\beta)i} = c 2^{-(d+1)} 2^{i(d-\alpha\beta)}.   \end{equation}
		Since $d>\alpha\beta$, this parameter tends to infinity as $i$ becomes large.
		By considering the admissible path $\gamma (s) \equiv \mathbf{0}$ for $0 \leq s \leq \frac{\tau_i}{4}$ 
		we have that $W( \frac{\tau_i}{4} ,\mathbf{0})$ is stochastically greater than $\tau_i (N_i- \frac{1}{4} )$.
		This suffices to show the lack of tightness.
\end{enumerate}
\end{proof}

%	\end{enumerate}
%\end{proof}

\begin{remark}We have the following two remarks on workload $W(t,x)$, the proofs are extensions of arguments in the above proof, and we {omit} their details.
		\begin{enumerate} 
			\item In the proof of first case, we use the fact that the parameter $\beta$ is \textit{strictly} smaller than $d$, which implies the $d+1-\frac{\epsilon}{2}$-th moment of $\tau$ is infinite. In fact, with an extra amount of work, we can show that the model is unstable with the condition that $\tau$ has an infinite $d+1$-th moment, so that there are also unstable models with parameters $(\alpha,\beta)$ on the boundary of regions II,III, (i.e. $\beta=d$). Similarly, we can also \edtf{show that} the model is unstable with the condition that $R$ has an infinite $d$-th moment, when $\alpha =0$.
			\item  When \edtfff{$\alpha <0$}, the workload $W(t,\mathbf{0})$ is infinite almost surely, for any $t>0$. 
		\end{enumerate}
\end{remark}

\subsection{Reduction of Theorem \ref{thm: stability in V}} \label{subsec: stability in region V}
The purpose of this subsection is to reduce the proof of Theorem \ref{thm: stability in V} to proving corresponding results { for } a ``discrete" comparison process.  This discretization, as mentioned, permits the use of inductive arguments.  The reduced result, Theorem \ref{thm: existence model in IV} below,  is subsequently further reduced in the next section. That is, our objective is to reduce proving Theorem  \ref{thm: stability in V} to proving Theorem \ref{thm: existence model in IV}  below.

\ledd{Recall that $( \tau', R') \sim \tilde P$ means that $\tilde P$ is the law of $(\tau', R')$.} We say that law $\tilde{ \tilde P} $ {\it dominates}
$\tilde P $ if %in the sense that $(\tau, R)$ under the former is stochastically larger than under the latter law,:
$$
\exists \tau' \leq \tau \mbox{ and } R' \leq R \mbox{ so that }
( \tau', R') \sim \tilde P \mbox{ and } ( \tau, R) \sim \tilde {\tilde P}.
$$
The main tool which follows from Remark \ref{remark: estimates for further stable model} is 
\begin{proposition} \label{basicdom}
If $\tilde{ \tilde {P}} $ dominates $\tilde P$ and $\tilde {\tilde {P}}$ is stable, then so is $\tilde P$.
\end{proposition}

We now give a family of distributions $\tilde P^{\theta,A}$ for $\theta, A >0$ which always dominate a given $\tilde P $.
Given $\theta $,$A$ and $\tilde P \in V$ we choose the sequences {$i\geq 1$}  \begin{equation} \label{eq: concrete choice of S, T,0}
	S_i = 2^{(1+\delta)A(1+\theta)^{2i}}, \quad T_j =2^{\alpha A (1+\theta)^{2i+1}}.
	\end{equation} 
	%for some $A>0$, and $i\geq 0$, $j\geq 1$. 
We also choose $T_0=0$ for convention.  The value $\delta > 0$ is specified as a function of $\theta $ and $\tilde P $ below.  Given these sequences $\tilde P^{\theta,A}$ {denotes} the law of $(\tau', R')$ obtained from $(\tau,R) \sim \tilde P$ by specifying as in 2) of Remark \ref{remark: estimates for further stable model}.
Thus from Proposition \ref{basicdom}, we have
that 
 Theorem  \ref{thm: stability in V} is implied by the following:

\begin{proposition} \label{basicdom2}
For $\tilde P $ {in} region $V$, there exists $\theta $ sufficiently small and such that there exists $A$ sufficiently large so that 
$\tilde P ^{\theta, A}$ is stable.
\end{proposition}

It remains to discuss the value $\delta $.  We {also} introduce a fourth value $\kappa > 0$ which {plays} a role
in our analysis.
Consider a generic $\tilde P$ in the region V. \edtt{Notice that region V $=\{\alpha <1, \alpha \beta>d  \}$ is an open set, we can find parameters $\alpha_0,\beta_0$ with $\alpha_o\cdot \beta_o >d \geq 1$, $0<\alpha_o < 1 $, such that the job sizes $(\tau_o, R_o)$ satisfy moment conditions}
	\begin{equation} \label{eq: finite moments}
	\mathbb{E}\left[ R_o^{d+\alpha_o} + \tau_o^{1+\beta_o} \right]<\infty.
	\end{equation}\edtss{Although the critical moments $(\alpha_c,\beta_c)$ of{ $\tilde P $} may not have strict inequality as \eqref{eq: finite moments}, from the definitions \eqref{eq: critical moment}, we can always approximate $(\alpha_c, \beta_c)$ by $(\alpha_0,\beta_0)$.}
	
%	We first fix some parameters for the new system, and define job sizes $(S_i),(T_j)$. 
For $\theta>0$, we define $\alpha_1,\beta_1$ parametrized by $\theta$
	\begin{equation}\label{eq: new a,b}
	\alpha_1 = \frac{d+\alpha_o}{(1+\theta)^2}-d,\text{ and } \beta_1 = \frac{1+\beta_o}{(1+\theta)^2}-1.
	\end{equation}
	Since $\lim_{\theta \downarrow 0}  \alpha_1 \beta_1 = \alpha_o \beta_o >d$, we can find a $\theta >0$, such that the pair $(\alpha_1, \beta_1)$ is a point in region $V$, 
	\begin{equation}  \label{eq: new a,b point}
	0<\alpha_1 < \alpha_0 <1, \text{ and } \alpha_1\beta_1 >d.    
	\end{equation} \edtss{Then, we can define $\delta$ by}
	\begin{equation} \label{eq:new condtion}
		\delta = {  \min\left(\edts{\frac{\alpha_1\beta_1 -d}{d+\alpha_1}} , \frac{\theta}{3}\right)}>0,
	\end{equation} and take a {$\kappa <  \min\{ \frac{1}{8}, \frac{\delta}{4(1+\delta)} \}$}. \edts{It follows immediately that $\theta> 3\delta> \frac{3\kappa} {1-\kappa} $, and therefore, we have}
	\begin{equation}  \label{eq: theta, kappa relation}
	\theta > \frac{\edts{3}\kappa }{1-\kappa}.
	\end{equation}
%\iffalse
% Lastly, choose $(S_i)_{i\geq 0}, (T_j)_{j\geq 0}$ of the form %\eqref{eq: concrete choice of S, T,0}
%	\begin{equation} \label{eq: concrete choice of S, T,0}
%	S_i = 2^{(1+\delta)A(1+\theta)^{2i}}, \quad T_j =2^{\alpha A (1+\theta)^{2i+1}}
%	\end{equation} 
%	for some $A>0$, and $i\geq 0$, $j\geq 1$. We also choose $T_0=0$ for convention. 
%
%\fi

By the Markov Inequality and \eqref{eq: finite moments},  we have for each job in the new system (i.e,
$(\tau,R)$ {chosen from law $\tilde P ^{\theta, A}$)} for $i\geq \edtss{1}, j\geq 1$
	\begin{align*}     
	P(R = S_{i+1} ) \leq P(R_o \geq S_{i} ) \leq  \mathbb{E}[R_o^{d+\alpha_0}] S_{i}^{-(d+\alpha_0)},
	\notag\\
	P(\tau = T_{j+1} ) \leq P(\tau_o \geq T_{j} ) \leq  \mathbb{E}[\tau_o^{1+\beta_0}] T_{j}^{-(1+\beta_0)},
	\end{align*}
	which, from the choices of $(\alpha_1, \beta_1)$ and $S_{i+1} = S_i^{(1+\theta)^2},T_{j+1} = T_j^{(1+\theta)^2} $, are bounded by
	\begin{align*}     
	\mathbb{E}[R_o^{d+\alpha_0}] S_{i}^{-(d+\alpha_0)} =  \mathbb{E}[R_o^{d+\alpha_0}] S_{i+1}^{-\frac{d+\alpha_0}{(1+\theta)^2}} \leq c_1 S_{i+1}^{-(d+\alpha_1)},
	\notag\\
	\mathbb{E}[\tau_o^{1+\beta_0}] T_{j}^{-(1+\beta_0)} =  \mathbb{E}[\tau_o^{1+\beta_0}]  T_{j+1}^{-\frac{1+\beta_0}{(1+\theta)^2}}  \leq  c_2 T_{\edtss{j+1}}^{-(1+\beta_1)},   
	\end{align*} for some $c_1, c_2>0$. 

Therefore, the distribution of $(\tau,R)$ in the new system (i.e. $\tilde {P^{\theta,A}}$)satisfies, for $i\geq 1, j\geq 1$
	\begin{align}     \label{eq: spatial dist}
		P(R = S_{i} ) \leq c_1 S_{i}^{-(d+\alpha_1)},
		\\
		P(\tau = T_{j} )  \leq   c_2 T_{j}^{-(1+\beta_1)},  \label{eq: temporal dist}
\end{align} for some $c_1, c_2>0$. \ledd{Note while $c_1, c_2$ depends on $A$, exponents $\alpha_1, \beta_1$ do not.} 

Our final reduction is just to suppose that in inequalities \eqref{eq: spatial dist} and  \eqref{eq: temporal dist} we \ledd{may take} equality.  Given law $\tilde P^{\theta,A}$ satisfying the above inequalties, it is immediate that by adding mass at points $(0,S_{{j}}) $ and $(T_j,0)$, we can construct a measure (of mass at least $1$) $\mu $  so that we have for $\mu $ equality in  \eqref{eq: spatial dist} and  \eqref{eq: temporal dist}.   Let $\tilde {\tilde {P}}^{\theta,A} $
be $\mu $ normalized to be a probability.
We now simply note that from basic Poisson process properties that for each $\lambda > 0$ the system $\tilde {\tilde {P}}^{\theta,A}$ at rate $|\mu| \lambda $ dominates system $\tilde P$ at rate $\lambda$ with obvious meaning of dominates in this context.  Thus we have that to show \ledd{Proposition \ref{basicdom2}, and therefore} Theorem  \ref{thm: stability in V} it suffices to show

\begin{theorem}
 \label{thm: existence model in IV}
Given $\tilde P $ in $V$, let   $\theta $ >0 be such that for { all} $A>0$ the law  $\tilde {\tilde {P}}^{\theta,A}$is in region $V$ and so that for $(\tau,R) \sim \tilde {\tilde {P}}^{\theta,A}$,
$P(R = S_{i} ) =c_1 S_{i}^{-(d+\alpha_1)},$ and $P(\tau = T_{j} ) =   c_2 T_{j}^{-(1+\beta_1)},$
for some $c_1$ and $c_2$.
Then for $A$ sufficiently large {$\tilde {\tilde {P}}^{\theta,A}$} is stable.
\end{theorem}

\section{A Stable Discretized Model, and Its Growth Estimates}\label{sec:growth}
{In this section, we focus on the new discrete system introduced in Theorem \ref{thm: existence model in IV}. Our main goal of this section is to reduce the proof of Theorem \ref{thm: existence model in IV} to a technical result, Proposition \ref{prop: precise estimate}.}
%  The proof of Theorem \ref{thm: existence model in IV} (modulo technical results) is provided at the end of this section.

%\subsection{Discrete Job Sizes, and Some Properties} 
We start with some conditions satisfied by the job sizes for $\tilde{\tilde{P}}^{\theta,A},$ when $A$ is sufficiently large. The specific form \eqref{eq: concrete choice of S, T,0} of job sizes is convenient for presenting the estimates and proofs in this section. In principle, there are many other choices for $(S_i ,T_j)$, as long as the system belongs to region V. The form \eqref{eq: concrete choice of S, T,0} of $(S_i,T_j)$ and a large enough $A$ imply conditions  \eqref{eq: well-separated} -- \eqref{eq: useful ones} below, and they are in the assumptions of two techinical propositions, Proposition \ref{prop: formal estimate} and Proposition \ref{prop: precise estimate}. Again, these conditions are not necessary for the stability, but they aid our proofs. We will discuss their roles after stating them.
  
  { For large $A$, the sequences} $(S_i), (T_j)$ satisfy the following conditions:
\begin{enumerate}
	 \item 	The choice of exponents implies that for $i>1$ \begin{equation}\label{eq: convenient exponents}
	 	S_i^{\alpha(1+\theta)} =T_i^{1+\delta}. 
	 \end{equation}
		\item Scales $ \Delta t_{2i-1}=S_i^\alpha\geq 1,\Delta t_{2i}= T_i^{1+\delta}\geq 1$ are well separated: \edts{for all $i\geq 1$} 
		\begin{equation} \label{eq: well-separated}
		\rcomm{\varrho}^{i} {\Delta t_i}^{1+\kappa} < \rcomm{\varrho}^{-(i+1)} {\Delta t_{i+1}}^{1-\kappa}, 
		\end{equation} \edtss{where $ \varrho = 10 +  \frac{d}{\alpha\kappa}$. 
		In fact, we can choose for each $i$,}
		\begin{equation}\label{eq: choices of l_i, r_i}
		l_i = \varrho^{-i}\Delta t_i^{-\kappa}, r_i = \rcomm{\varrho}^{i}\Delta t_i^{\kappa},
		\end{equation}\ledd{so that we have an increasing sequence of times in  $\mathbb{R}_+$} 	
		\begin{equation}\label{eq: partition of real line}
				0<  l_1  \Delta t_{1} < r_1\Delta t_{1} < l_2 \Delta t_{2} < r_2 \Delta t_{2} <\ldots .  
				\end{equation} To avoid confusion, we point out that the symbol $\Delta t_i$ does not mean the increment of time $t_{i+1}-t_i$, but a scale corresponding to $i$.
			
	We show \eqref{eq: well-separated} when $i$ is odd, $i=2j-1$. Indeed, there is an $A>0$ such that for all $j\geq 1$, 
		\[(\Delta t_{2j-1})^{\kappa}= 2^{A\kappa\alpha(1+\delta)(1+\theta)^{2j}}>2^{2 A\kappa\alpha\theta(1+\delta)j}> \rho^{4j-1}.\] Then by \eqref{eq: theta, kappa relation}, we have that $1+\theta> \frac{1+2\kappa}{1-\kappa}$. Therefore, from \eqref{eq: convenient exponents},
\[ \rcomm{\varrho}^{4j-1} {\Delta t_{2j-1}}^{1+\kappa} < {\Delta t_{2j-1}}^{1+2\kappa} < {\Delta t_{2j-1}}^{(1+\theta)(1-\kappa)}={\Delta t_{2j}}^{(1-\kappa)} ,  \] which implies \eqref{eq: well-separated}.	
Similarly, we can show the case when $i$ is even, and we omit the proof.
		
		In fact, \eqref{eq: well-separated} also implies the following two conditions (when $A$ is large), see \eqref{eq: no influence from upper} \eqref{eq:summability} below and a short proof of them at the end of \ledd{point 4}.
		%	  Alternatively, this is the same as 
		%	  \[l_{2i-1}= (1+r)^{-i}, r_{2i-1} = S_i^{\frac{\alpha}{8}}, l_{2i}=(1+r)^{-i}T_i^{-\frac{\alpha}{8}}, r_{2i} =T_i^{\frac{\delta}{8}}
		%	   \]	
		\item Higher scale has little influence over lower scales: \edts{for all $i\geq 1$} 
		\begin{equation} \label{eq: no influence from upper}\edts{ 
		\sum_{j>0}\varrho^{\edtt{2j}} \Delta {t_{2(i+j)-1}}^{-(1-\kappa)} <  \Delta{t_{2i}}^{-1}}.
		\end{equation}

		\item Summability: \edts{for each $C>0$,}
		\begin{equation}\label{eq:summability}
		\edts{\sum_{i>1}\exp(-C{\Delta t_{i}}^{\kappa} ) \ln (\Delta t_{i})< \infty}. 	  
		\end{equation}
 \eqref{eq: no influence from upper} and \eqref{eq:summability} { follow }from \eqref{eq: well-separated}. In fact, \eqref{eq: well-separated} implies that for all $i\geq 1$, \[  \Delta t_{i}^{-1} > \rho^{2i+1} \Delta t_{i+1}^{-(1-\kappa)}>\rho^{3} \Delta t_{i+1}^{-1}.  \] Therefore, we have that for any $j>0$,
		\[  \rho^{-j}\Delta t_{2i}^{-1} > \rho^{2j} \Delta t_{2(i+j)-1}^{-(1-\kappa)}. \] Summing over $j$, we can get \eqref{eq: no influence from upper}.
	
		Similarly, From \eqref{eq: well-separated}, we get that $\Delta t_i \geq \rho^{3i}.$ For any $C>0$, there is an $M>0$ such that the function $\exp(-C x^{\kappa} ) \ln (x)$ is decreasing in $x$ on $[M,\infty)$. We can get \eqref{eq:summability} from the inequality $\sum_{i>1}\exp(-C{\rho^{3i\kappa}} ) \ln (\rho^{3i})< \infty$.
		
		\item Leading order: \edts{for all $i\geq 1$} 
		\begin{equation} \label{eq: leading order condition}
		\edts{\left(S_i^{1-\alpha}\Delta t_{2i}\right)^{d}\Delta t_{2i} T_i^{-(1+\beta)}<1}
		\end{equation}
		 This is a consequence of $\alpha \beta >d $, and the choice of $\delta$, see \eqref{eq:new condtion}. In view of \eqref{eq: convenient exponents}, we have that $S_i < T_i^{\frac{1+\delta}{\alpha}}$. Therefore, we bound the left hand side of \eqref{eq: leading order condition} (strictly) by
			\begin{equation*}  \label{eq:leading order 0}
				T_i^{\left((\frac{1-\alpha}{\alpha} + 1)\right)d+1)(1+\delta) -(1+\beta)} = T_i^{\frac{d-\alpha\beta + (d+\alpha)\delta}{\alpha}}\leq 1 .  \end{equation*}
	\item		Also, when $A$ is large, we have inequalities for $(S_i)$, for all $k \geq 1$,
			\begin{equation} \label{eq: useful ones} 
				10\sum_{i < k} S_i^{1-\alpha} <S_k^{1-\alpha}, \quad 10\sum_{i>k} S_i^{-\alpha} < S_{k}^{-\alpha}.
			\end{equation} 
	\end{enumerate} 

%	 It is not hard to verify that the sequences $(S_i), (T_j)$ in form \eqref{eq: concrete choice of S, T,0} satisfy \eqref{eq: well-separated} and therefore \eqref{eq: choices of l_i, r_i}, \eqref{eq: no influence from upper}, and \eqref{eq:summability} hold. Also, a large enough constant $A$ ensures \ref{eq: useful ones} for all small $k$ (while \ref{eq: useful ones} is obvious for large $k$). Lastly, we will verify \eqref{eq: leading order condition}.

\edtss{Before going into the major \ledd{bounds} in the proof of Theorem \ref{thm: existence model in IV},} we briefly explain the roles of conditions \eqref{eq: well-separated} -- \eqref{eq: useful ones} in the major steps. Conditions \eqref{eq: well-separated} and \eqref{eq:summability} are used directly in the proof of Theorem \ref{thm: existence model in IV}, see subsection \ref{subsec: proof of main theorem} below. In general, \eqref{eq: well-separated} allows us to look at \edtss{different time intervals corresponding to different time scales} and obtain estimates according to these scales, while \eqref{eq:summability} allows us {to get a finite sum from these estimates, and to apply a Borel-Cantelli type of argument, see Lemma \ref{lm: BC lemma} below.} Conditions \eqref{eq: useful ones} enable us to estimate the probability to connect two {time-space} points, \edtf{which is} used throughout sections \ref{sec: boundary process, modified paths}, \ref{sec: leading order}. In particular, condition \eqref{eq: useful ones} \ledd{has} nothing to do with temporal workload, and it controls spatial growth for admissible paths for $W(t,\edtt{\mathbf{0}})$. Condition \eqref{eq: leading order condition} \edtss{is the cornerstone to control the growth of the maximal numbers of jobs, see \eqref{eq: temporal number} and \eqref{eq: spatial number} below, and it plays a central role in Proposition \ref{lm: leading order}}. This condition implies that when jobs of spatial sizes larger than $S_j$ are not involved, the density of jobs of temporal size $T_j$ along any admissible paths, is at most $T_j^{-(1+\delta)}$, for some $\delta>0$. Condition \eqref{eq: no influence from upper}, on the other hand, ensures that jobs of sizes larger than $S_j$ { do} not affect the leading order $T_j^{-(1+\delta)}$ for the density of jobs of temporal size $T_j$ along any admissible paths, see Proposition \ref{lm: no influence}. It is with \eqref{eq: no influence from upper} that we bound a family of tailored exponential moments, see Corollary \ref{cor: exponential moments} below, and derive from them the probabilistic estimates in Proposition \ref{prop: precise estimate}.

\subsection{Growth Estimates}	\label{sub: growth}
{The job sizes are discrete for $\tilde{\tilde{P}}^{\theta,A}$, so we can estimate the load of an admissible path by contribution from different job sizes.} For any fixed admissible path $\gamma_{0,t}$, we denote by $n_j(\gamma_{0,t})$ the number of jobs of the temporal size $T_j$ that $\gamma_{0,t}$ intersects (in the sense of \eqref{eq: intersection}), and denote by $m_i(\gamma_{o,t})$ the number of jobs of spatial size $S_i$ that $\gamma_{0,t}$ intersects:
\begin{align}
n_j(\gamma_{0,t}) :=& \text{ number of jobs of temporal size } T_j \text{ intersected by }  \gamma_{0,t},    \label{eq: temporal number} \\
m_i(\gamma_{0,t}) :=& \text{ number of jobs of spatial size } S_i \text{ intersected by } \gamma_{0,t}.   \label{eq: spatial number}
\end{align}	
The load of an admissible path can be written as 
{\[U(\gamma_{0,t})= \sum_j n_j(\gamma_{0,t}) T_j. \]}
{Taking the supremum over all admissible paths with the initial point $\gamma_{0,t}(0)=\mathbf{0}$, we get three types of maximal quantities, which are all superadditive in time $t$,
\begin{equation*}
	\sup U(\gamma_{0,t}),	\sup n_j(\gamma_{0,t}), \text{and} \sup m_j(\gamma_{0,t}).
\end{equation*} The technical propositions involve bounding the growth of these quantities.
}

%%%%TBC
The main ingredient towards Theorem \ref{thm: existence model in IV} is the following proposition on the maximal load up to time $t$, $\sup U(\gamma_{0,t})$. It says that $\sup U(\gamma_{0,t})$ grows at most linearly in time $t$ with a high probability. \rcom{In view of \eqref{eq: partition of real line}, the probabilistic estimate is related to different time scales on different time intervals. On each time interval \edtss{larger than the time scale \edtss{$\Delta t_{2i}$}}, the probability that {$\sup U(\gamma_{0,t})$} grows faster than $c\cdot t$ is exponentially small in the time scale $\Delta t_{2i}$, up to a term which is proportional to the ratio of $t$ and the next time scale $\Delta t_{2(i+1)}$\edtss{; on the time interval close to the time scale $\Delta t_{2i}$, we can modify the event slightly.}}
%%and we can state the formal upper bound estimates for $n_i, m_i$ in forms of union bounds.
% TBD
\begin{proposition}\label{prop: formal estimate} Recall $\kappa < \min\{\frac{1}{8}, \frac{\delta}{4(1+\delta)}\}$. \ledd{Suppose the assumption of Theorem \ref{thm: existence model in IV} hold, and} two sequence of numbers $(S_i)$ $(T_j)$ satisfy conditions \eqref{eq: well-separated}--\eqref{eq: useful ones}, for $\Delta t_{2i-1}=S_i^{\alpha}$, $\Delta t_{2i}=T_i^{1+\delta}$, and $l_i,r_i$ defined by \eqref{eq: choices of l_i, r_i}, for all $i \in \mathbb{N}$. Then, there are positive constants $c$, $C$, $C'$, and $\lambda_d$ such that if $\lambda < \lambda_d$, then for all $i\geq 1$, 
	\begin{equation}\label{eq: formal estimate}
	P\left( \sup U(\gamma_{0,t})> c\cdot t  \right)	< \begin{cases} \exp(-C {\Delta t_{2i}}^{\kappa}) +  C' {\Delta t_{\edtff{{2i+2}}}}^{-1} t &, \text{ if } l_{2i} {\Delta t_{2i}} \leq t < r_{2i} {\Delta t_{2i}} \\
	\exp(-C {\Delta t_{2i}}^{-1}\cdot t ) +  C' {\Delta t_{\edtff{{2i+2}}}}^{-1} t &, \text{ if } r_{2i} {\Delta t_{2i}}\leq t < l_{2i+2} \edtff{{\Delta t_{2i+2}}}
	\end{cases},
	\end{equation}
	where the supremum is taken over all admissible paths with the initial point $\gamma_{0,t}(0)=0$.
\end{proposition} 
Before proceeding, we remark that we can get an upper bound for $C'$, but due to our choices of $r_i$, $l_i$ and time scales $\Delta t_i$, \eqref{eq: formal estimate} converges to $0$ fast enough regardless of the values of $C'$. In fact, we use \eqref{eq: formal estimate} to prove a Borel-Cantelli type of lemma, see Lemma \ref{lm: BC lemma} \edtf{below}.

Notice that the time scales have been standardized{, in the sense that} the scale ${\Delta t_{2i} = T_i^{1+\delta}}$ is a typical time of seeing one arrival of  $T_i$ along the maximal path. A more precise statement is Proposition \ref{prop: precise estimate}. Let $A_{t,j}(c')$ be the event that the maximal number for jobs of temporal size $T_j$ doesn't exceed $c' T_j^{-(1+\delta)}\cdot t$ by time $t$ for any admissible path on $[0,t]$ with the initial point \ledd{$(0,\mathbf{0})$},
	\begin{equation} \label{eq: overgrowth of temporal jobs}
	A_{t,j}(c'):= \left\{ \sup_{\substack{\gamma_{0,t}(0)=\edtf{\mathbf{0}},\\ \text{$\gamma_{0,t}$ admissible}}} n_j(\gamma_{0,t}) \leq c'T_j^{-(1+\delta)}t \right\},
	\end{equation}  
	and let $B_{i}(t)$ be the event that no jobs of spatial size $S_j$, \edtfff{$j\geq i+2$} occur along any admissible path on $[0,t]$ with the initial point \ledd{$(0,\mathbf{0})$},
	\begin{equation} \label{eq: overgrowth of spatial jobs}
	B_{i}(t):= \left\{ \sup_{\substack{\gamma_{0,t}(0)=\edtt{\mathbf{0}},\\ \text{$\gamma_{0,t}$ admissible}}} m_j(\gamma_{0,t}) =0, \text{ for all $j\geq i+2$ }\right\}.
	\end{equation}  
	
	For every $t$ in a fixed time interval $ [l_{2i} \Delta t_{2i},l_{2i+2} \Delta t_{2i+2}]$, we consider events described by cases in Proposition \ref{prop: precise estimate} below, and estimate their probability, which are in general small. We then upper bound the probability of the event that $A_{t,j}(c')$ occurs for all $j$. Small modification is required when $t$ is close to  to $\Delta t_{2i}$. In this case, the ratio ${\Delta t_{2i}}^{-1}t$ is like a constant, and we consider the event $A^c_{t,i}(c'\Delta t_{2i}^{2\kappa})$, which also has a small probability when $B_i(t)$ occurs, see \eqref{eq: unusual temporal job} below. In general, we can chose $c =   c' \sum_{j=1}^{\infty} T_j^{-( \delta-2\kappa )}<\infty$, and prove Proposition \ref{prop: formal estimate} by \eqref{eq: detail estimate} in Proposition \ref{prop: precise estimate}. So we omit the proof of Proposition \ref{prop: formal estimate}. 

\begin{proposition}\label{prop: precise estimate} 
	Under the assumption of Proposition \ref{prop: formal estimate}, there are positive constants $c'$, $C$, $C'$ and $\edts{\lambda_d<1}$, such that, if $\lambda<\lambda_d$, for any $t\in [l_{2i} \Delta t_{2i},l_{2i+2} \Delta t_{2i+2})$ we have the following estimates on maximal numbers of jobs:
	\begin{enumerate}
		\item $B_i(t)$ occurs with a high probability{:}
		\begin{equation}\label{eq: no large spatial job}
			P\left( B_i(t) ^c\right)  \leq   C' {\Delta t_{2i+2}}^{-1} t  \leq C' l_{2i+2} ;
		\end{equation}
		
		\item if $j>i$,
		\begin{equation} \label{eq: no large temporal job}
		P\left( A^c_{t,j}(0) \text{ and } B_i(t)\right) < C' {\Delta t_{2j}}^{-1} t  ;
		\end{equation} 
		
		\item if $j\leq i$, 
		\begin{equation} \label{eq: usual temporal job}
		P\left( A^c_{t,j}(c') \text{ and } B_i(t)\right) < \exp\left(-C \kappa^{i-j+2} {\Delta t_{2j}}^{-1} t\right)  ;
		\end{equation} 
		
		\item in the case when $i=j$, and when $t< r_{2i}\Delta t_{2i}$,
		\begin{equation} \label{eq: unusual temporal job}
		P\left( \edtfff{A^c_{t,i}}(c'{\Delta t_{2i}}^{2\kappa})  \text{ and } B_i(t)\right) < \exp(-C  \Delta t_{2i}^{\kappa})  ;
		\end{equation} 
	\end{enumerate}
Furthermore, we {bound} the event that $A_{t,j}(c'{\Delta t_{2j}}^{2\kappa})$ occurs for all $j$,
\begin{align} \label{eq: detail estimate}
1- P&\left( A_{t,j}(c'{\Delta t_{2j}}^{2\kappa})  \text{ for all $j$ } \right)
\notag
\\
 <& \begin{cases} \exp(-C {\Delta t_{2i}}^{\kappa}) +  C' {\Delta t_{\edtff{{2i+2}}}}^{-1}\cdot  t &, \text{ if } l_i {\Delta t_{2i}} \leq t < r_i {\Delta t_{2i}} \\
\exp(-C {\Delta t_{2i}}^{-1}\cdot t ) +  C' {\Delta t_{\edtff{{2i+2}}}}^{-1} \cdot t &, \text{ if } r_i {\Delta t_{2i}}\leq t < l_{i+1} {\Delta t_{\edtff{{2i+2}}}}
\end{cases}
\end{align} 
	\end{proposition} 
The proof of Proposition \ref{prop: precise estimate} is postponed to section \ref{sec: proof of main estimate}. %\ref{sec: large jump}.  
To aid the proof, we introduce some tools in section \ref{sec: boundary process, modified paths}, and use them to estimate exponential moments of the maximal numbers in section \ref{sec: leading order}. By \edtf{the} Markov Inequalities, we derive estimates in Proposition \ref{prop: precise estimate} from exponential moments. 

\subsection{A Borel-Cantelli Type of Lemma}\label{sbsec: BC lemma}
The next lemma shows that we can get stability from a Borel-Cantelli type of estimate, see \eqref{eq: BC lemma}, \ledd{and therefore, showing Theorem \ref{thm: existence model in IV} is reduced to Proposition \ref{prop: formal estimate}.} The proof is elementary, and it comes from extensions of admissible paths, and monotonicity of $\tilde{W}(t,\mathbf{0})$ in $t$.  A similar argument can be found in the proof of Proposition 14 \cite{18FKM}. 
%\edts{ Recall that $\lambda_d<1$ is from Proposition \ref{prop: precise estimate}.} 
\begin{lemma}\label{lm: BC lemma}
	Suppose there exist positive constants {$a>0$, $b< \frac{\edts{1}}{1+a}$} and a diverging sequence of times \edtff{$(u_n)$} such that \edtff{$u_n \leq u_{n+1} \leq (1+a)u_n$}, and 
	\begin{equation} \label{eq: BC lemma}
	\sum_n P\left( \sup U(\gamma_{0,\edtff{u_n}})> b \cdot \edtfff{u_n}  \right)  < \infty,  
	\end{equation} where the supremum is over all admissible path with the initial point {$(\edtt{0,\mathbf{0}})$}.
	Then $\{\tilde{W}(t,\edtt{\mathbf{0}})\}_{t\geq 0}$ is tight.
\end{lemma}
\begin{proof} Firstly, for $s'<s$, we can always extend an admissible path $\gamma_{0,s'}$ on $[0,s']$ to a new one $\gamma_{0,s}$ on $[0,s]$ by letting
	${\gamma_{0,s}(u) =
		\gamma_{0,s'}(u)}$, ${0\leq u\leq s'}$, and
	$\gamma_{0,s'}(s'), \text{ otherwise.} $ {Therefore,} $U(\gamma_{0,s}) \geq U(\gamma_{0,s'})$.
	By \eqref{eq: BC lemma}, we get that for any $\epsilon>0$, there exists an $n_\epsilon$, such that
	\[P\left(  U(\gamma_{0,u_n})> b \cdot \edtfff{u_n}, \text{ for some } n\geq n_\epsilon, \text{ and some admissible path } \gamma_{0,u_n} \text{ with } \gamma_{0,u_n}(0)=\edtt{\mathbf{0}}  \right)  < \epsilon.\] 
	Then by extending admissible paths, we have 
	\begin{equation*}P\left(  U(\gamma_{0,s})> b(1+a) \cdot s, \text{ for some } s\geq u_{n_\epsilon}, \text{ and some admissible path } \gamma_{0,s} \text{ with } \gamma_{0,s}(0)=\edtt{\mathbf{0}}  \right)  < \epsilon.
	\end{equation*}
{Since} $V\left(\gamma_{0,s}\right) = U(\gamma_{0,s})-s$, and $b(1+a)-1<0$, we get
	\begin{equation}\label{eq:finiteness2}P\left(  V(\gamma_{0,s})> 0, \text{ for some } s\geq u_{n_\epsilon}, \text{ and some admissible path } \gamma_{0,s} \text{ with } \gamma_{0,s}(0)=\edtt{\mathbf{0}}  \right)  < \epsilon.
	\end{equation}	
	We also get a similar statement from \eqref{eq: BC lemma} and extending admissible paths, 
	\[P\left(  U(\gamma_{0,u_{n_\epsilon}})> b \cdot u_{n_\epsilon}, \text{ for some admissible path } \gamma_{0,u_{n_\epsilon}} \text{ with } \gamma_{0,u_{n_\epsilon}}(0)=\edtt{\mathbf{0}}  \right)  < \epsilon,\]
{which implies that} 
		\begin{equation} \label{eq:finiteness}
		P\left( \tilde{W}(\edtff{u_{n_\epsilon}},\edtt{\mathbf{0}})> (b+1)u_{n_\epsilon} \right)   < \epsilon.
		\end{equation}
	
	Now we prove tightness of {$\left(\tilde{W}(t,\edtt{\mathbf{0}})\right)$.} If $t>\edtff{u}_{n_\epsilon}$, by \edtff{\eqref{eq:finiteness2},\eqref{eq:finiteness}}, and the definition of $\tilde{W}(t,\edtt{\mathbf{0}})$, we obtain that there is an $N_\epsilon= (b+1)\edtff{u}_{n_\epsilon}>0$ such that		\begin{align*}
		P\left(  \tilde{W}(t,\edtt{\mathbf{0}})>N_\epsilon \right) \leq& P\left(  \tilde{W}(\edtfff{u}_{n_\epsilon},\edtt{\mathbf{0}})>N_\epsilon \right)  \\
		+ P&\left(  V(\gamma_{0,s})>N_\epsilon, \text{for some }s \geq \edtff{u}_{n_\epsilon}, \text{ and some admissible path } \gamma_{0,s} \text{ with } \gamma_{0,t_s}(0)=\edtt{\mathbf{0}} \right) \\
		<& 2\epsilon; 	
		\end{align*} 
		If $t\leq \edtff{u}_{n_\epsilon}$, from the fact that $\tilde{W}(t,\edtt{\mathbf{0}})$ is increasing in $t$, we get that 
		\[P\left(  \tilde{W}(t,\edtt{\mathbf{0}})>N_\epsilon \right) \leq P\left(  \tilde{W}(\edtff{u}_{n_\epsilon},\edtt{\mathbf{0}})>N_\epsilon \right)<\epsilon \]
	We conclude that $(\tilde{W}(t,\edtt{\mathbf{0}}))_{t\geq 0}$ and $(W(t,\edtt{\mathbf{0}}))_{t\geq 0}$ are tight.
\end{proof}

\subsection{Proof of Theorem \ref{thm: existence model in IV} given Proposition \ref{prop: precise estimate}}\label{subsec: proof of main theorem}
Now we prove Theorem \ref{thm: existence model in IV} by assuming Proposition \ref{prop: precise estimate}, and therefore Proposition \ref{prop: formal estimate}.

\begin{proof} (Theorem \ref{thm: existence model in IV}) {In view of Lemma \ref{lm: BC lemma}, we only need to find a sequence $(u_n)$ and to verify \eqref{eq: BC lemma} for some $b<1$.} Let $a>0$, we define a sequence of $(u_n)$ inductively:
	\edt{
	\begin{equation}
	\begin{cases} \edtff{u_1} =& l_1\cdot \Delta t_1\\
	\edtff{u_{n+1}} =&  \min \{(1+a)\edtff{u_n}, l_k\cdot \Delta t_k,:  l_k\cdot \Delta t_k > \edtff{u_n} \text{, for some $k$}\}
	\end{cases}.
	\end{equation}
} It is immediate that \edtff{$u_{i+1} =  (1+a)u_i$}, if $\edtff{u_{i+1}} \neq l_k\cdot \Delta t_k$ for any $k$.
	
	{From} Proposition \ref{prop: formal estimate}, there exists positive constants $c,C,C'$ such that, for any $k$
\edt{
	\begin{align*}
	\sum_{l_k\cdot \Delta t_k \leq \edtff{u_i} < r_k\cdot \Delta t_{k}} P&\left( \sup U(\gamma_{0,t_i})> c\cdot \edtff{u_i}  \right)	< \sum_{l_k\cdot \Delta t_k \leq \edtff{u_i} < r_k\cdot \Delta t_{k}} \left(\exp(-C \edts{\Delta t_{k}}^{\kappa}) +  C' {\Delta t_{k+1}}^{-1} \edtff{u_i} \right).
	\end{align*}
	By summing according to the constant $\exp(-C \edts{\Delta t_{k}}^{\kappa})$, and other terms linear in $\edtff{u_i}$, we rewrite the previous sum as
	\[\exp(-C {\Delta t_k}^{\kappa}) \cdot \abs{ \{\edtff{u_i}: l_k\cdot \Delta t_k \leq \edtfff{u_i} < r_k\cdot \Delta t_{k} \}   }    + C' {\Delta t_{k+1}}^{-1} \sum_{l_k\cdot \Delta t_k \leq \edtff{u_i} < r_k\cdot \Delta t_{k}} \edtff{u_i}. \]
\edts{Since} $(\edtff{u_i})$ grows geometrically in $i$ on the interval $[l_k\cdot \Delta t_k,r_k\cdot \Delta t_k]$, $\edtff{u_{i+1}} = (1+a)\edtff{u_i}$, we get the sum bounded by
	\[\exp(-C {\Delta t_k}^{\kappa}) \left( \frac{\log(r_k) -\log(l_k)}{\log(1+a)} +1\right) + C' \frac{r_k\cdot \Delta t_k}{\Delta t_{k+1}} \frac{1}{a},\] which by \eqref{eq: well-separated} and \eqref{eq: choices of l_i, r_i}, is bounded by
\begin{align}\frac{2}{\log(1+a)} \exp(-C {\Delta t_k}^{\kappa}) \log (\varrho^k \Delta t_k)  + \exp(-C {\Delta t_k}^{\kappa})  + C' \varrho^{-(k+1)}. \label{eq: inequality1}
	\end{align}}
 Similarly, for any k,
	\begin{align}
	\sum_{r_k\cdot \Delta t_k \leq \edtff{u_i} < l_{k+1}\cdot \Delta t_{k+1}} P&\left( \sup U(\gamma_{0,t_i})> c\cdot \edtff{u_i}  \right)	< \sum_{r_k\cdot \Delta t_k \leq \edtfff{u_i} < l_{k+1}\cdot \Delta t_{k+1}} \left( \exp(-C {\Delta t_k}^{-1}\cdot \edtff{u_i}) +  C' {\Delta t_{k+1}}^{-1} \edtff{u_i}  \right) 
	\notag \end{align}\edts{
Since ${u}_{i+1} =  {u_i}(1+a)$ when $u_{i+1} \neq l_k\cdot \Delta t_k$ for any $k$, the sequence $\left(\exp(-C {\Delta t_k}^{-1}\cdot u_i)\right)_{r_k\cdot \Delta t_k \leq u_i < l_{k+1}\cdot \Delta t_{k+1}}$ is bounded by a geometric sequence {with its first term} as $\exp(-C r_k)$ and a ratio at most $\exp(-C a\cdot r_k) \leq \exp(-C a)$. Therefore, we bound the sum by
\begin{align}
		&	\exp(-C r_k) \cdot \frac{1}{1-\exp(-C\edts{a})) }  + C' {\Delta t_{k+1}}^{-1} \sum_{r_k\cdot \Delta t_k \leq \edtff{u_i} < l_{k+1}\cdot \Delta t_{k+1}} \edtff{u_i}  \notag   \\
	<&  	\exp(-C r_k) \cdot \frac{1}{1-\exp(-C\edts{a}) }  + C' \frac{l_{k+1} }{a}. 
	\label{eq: inequality2}
	\end{align}}
 Summing over $k$ and by \eqref{eq:summability}, we get
	\[\sum_{n} P\left( \sup U(\gamma_{0,\edtfff{u_n}})> c\cdot \edtff{u}_n  \right) <\infty.\] By rescaling the arrival rate of the system {by a factor $r =\min\{ \frac{1}{2}, \frac{{\lambda_1}}{c(1+a)}\}$}, and choosing $u'_n = {r^{-1}} \cdot u_n $, we get the estimate \eqref{eq: BC lemma} with $b=r \cdot c {\leq}\frac{{\lambda_1}}{1+a}$. \edtss{Hence, the system is stable.} 
\end{proof}
\\Thus our proof is reduced to {proving} Proposition \ref{prop: precise estimate}.

\section{Connectivity and Spatial Growth}\label{sec: boundary process, modified paths}
In this section, we consider the growth of the range of the admissible paths, $\mathcal{R}_t$ (see \eqref{eq: domain of influence} below).  For this question  the temporal sizes $T_j$ play no roles, and results only involve $(S_n)$. For convenience, we {always} assume that $A$ is large enough so that \eqref{eq: useful ones} is valid for $(S_n)$. \ledd{As the job arrivals follow independent Poisson processes, for the law $\tilde{\tilde{P}}^{\theta,A}$ in Theorem \ref{thm: existence model in IV}, the arrival rate $\lambda_n$ for jobs with a (positive) spatial size $S_{n}$ is 
\begin{equation}\label{eq: arrival rate for jobs}
	\lambda_n = c_1 \lambda  S_n^{-(d+\alpha)},
\end{equation}
for some $c_1>0$; for convenience, we assume $c_1 =1$ for the rest of paper.} A major ingredient is to estimate the probability that two {time-space points $(0,\mathbf{0})$ and $(t,x)$ in $ \mathbb{R}_+ \times \mathbb{Z}^{d}$} are \textit{connected} via admissible paths involving only jobs of spatial sizes $S_j$, for $j\leq n$. For $t< S_n^\alpha$, we {show} that the probability decays geometrically in the spatial $l_\infty$ distance, 
\[  \abs{{(t,x)}}%_\infty
  \edtf{\equiv} \abs{x}%_{\infty}
,\] under the scale $S_{n}$, see equation \eqref{eq: main} below. The estimate \eqref{eq: main} is an analogue of \eqref{eq:stochastic bounds}. We {then } explain the similarity between these two inequalities. 
The value of this result is that it { cuts}  down the number of admissible \edtts{paths} that we must consider.

Before we make more precise statement, we introduce some definitions and assumptions. Similarly to subsection \ref{subsec: admissible path}, we use graphical constructions for the job arrivals. We define $n$-admissible paths, and a pair of $n$-connected {time-space} points by restricting to jobs of spatial sizes at most $S_n$:
\begin{enumerate}
	\item 
	An \textit{$n$-admissible path} is an admissible path \ledd{${\gamma^{(n)}: [u,t] \rightarrow \mathbb{Z}^d}$} \edtts{that} starts from {time-space} point $(u,x)$ and ends at $(t,y)$, and with the further restriction that only jobs of spatial sizes $S_j$, $j\leq n$ intersect it, in the sense of \eqref{eq: intersection}, \edtt{i.e., if $	\gamma^{(n)}(s)\neq \gamma^{(n)}(s-)$, then there is a job of spatial size $S_j$ with $j \leq n$ arriving at {$(s,x)$} with
	\begin{equation*}
		\gamma^{(n)}(s), \gamma^{(n)}(s-) \in B(x,S_j):=\{y:\abs{y-x}%_{\edts{\infty}}
			\leq  S_j\}, 
		\end{equation*}}
 for some $j\leq n$.\\
		 We also \ledd{use $\gamma^{(n)}_{u,t}$ or $\gamma^{(n)}_{u,t;x,y}$ when we emphasize that $\gamma^{(n)}$ is on the interval $[u,t]$ or it has end points $(u,x)$ and $(t,y)$.} 

\item \label{def: n-connected}
	Two {time-space} points $(t_1,x_1)$ and $(t_2,x_2)$ are \textit{n-connected} if there is an $n$-admissible path with starting and ending points $(t_1,x), (t_2,y)$, such that 
	\begin{equation}\label{eq: ending points distance}
	\abs{x-x_1}%_\infty
	, \abs{y-x_2}%_\infty 
	\leq \frac{S_{n}}{2}. 
	\end{equation}
\end{enumerate}
These definitions are to facilitate a discretization of the model at the scale $S_n$. We note that for a pair of $n$-connected points $(t_1,x_1)$ and $(t_2,x_2)$, there may not exist any $n$-admissible path $\gamma^{(n)}$ ``connecting" them in the sense that $\gamma^{(n)}(t_1)=x_1$ and  $\gamma^{(n)}(t_2)=x_2$. We also define the domain of influence $\mathcal{R}_t$ as a collection of points in {time-space} which are connected (or n-connected) to $(0,\mathbf{0})$ by admissible (or n-admissible) paths, 
\begin{align}\label{eq: domain of influence}
	\mathcal{R}_t &:= \left\{(x,s):\exists \text{ an admissible path $\gamma_{0,t}$, with $\gamma_{0,t}(0) = \mathbf{0}$, and $\gamma_{0,t}(s) = x$ for some $s\leq t$}     \right\}, \\
\edtts{	\mathcal{R}^{(n)}_t }&:= \left\{(x,s): \exists \text{ an n-admissible path $\gamma^{(n)}_{0,t}$, with $\gamma^{(n)}_{0,t}(0) = \mathbf{0}$, and $\gamma^{(n)}_{0,t}(s) = x$ for some $s\leq t$}     \right\}, \\
\edtts{	\mathcal{\tilde{R}}^{(n)}_t} &:= \left\{ (s,x): \text{ $(s,x)$ and  $(0,\mathbf{0})$ are $n$-connected for some $s\leq t$}     \right\}.
\end{align} It is obvious that $\mathcal{R}_t =  \bigcup_{n} 	\mathcal{R}^{(n)}_t \edtts{   \subset\bigcup_{n}	\mathcal{\tilde{R}}^{(n)}_t}  $.

Since the model is translation invariant, we \ledd{need only consider} the probabilities for two {time-space} points $(0,\mathbf{0})$ and $(t,x)$,
\begin{equation}\label{eq: n-connectivity}
p_n(t,x) :=  P\left(\text{$(t,x)$ and $(0,\mathbf{0})$ are n-connected}\right). 
\end{equation} 

In the case when $d=1$, we { use} \eqref{eq:stochastic bounds} and large deviation estimates to show that $p_n(t,x)$ decays geometrically in $x$ under the scale $S_{n}$ when $t\leq S_{n}^{\alpha}$. In fact, in any dimension $d\geq 1$, {it follows } from Proposition \ref{lm: d>1 case} that this {remains} true. 
\begin{proposition}\label{lm: d>1 case}
	Recall that $(S_n)$ satisfy \eqref{eq: useful ones} {when $A$ is large enough}. 
	 There are positive constants depending on $\lambda$, $q_n(\lambda)<1$, so that 
		for any $t\leq S_{n}^{\alpha}$ and positive integer $i$, when ${(i-\frac{1}{2})S_{n}  < \abs{x}%_\infty 
					< (i+\frac{1}{2})S_{n} }$
		\begin{equation}\label{eq: main}
			p_n(t,x)  \leq q_n(\lambda) ^{i-1}.
		\end{equation}
	Moreover, we \ledd{can} choose the upper bounds $q_n(\lambda)$ so that they satisfy the following properties: 
	\begin{itemize}
	\item for all $n>0$,
	\begin{equation}\label{eq:recursive formula}
		q_{n+1}(\lambda) = c_1 \lambda + \left(C_d \cdot q_{n}(\lambda)\right)^{\frac{5S_{n+1}}{8S_n}},
		\end{equation} for some positive constants $c_1, C_d$ depending only on the dimension $d$,
		\item
		$\exists \lambda_d>0$ and $c_d < \infty $ depending \edt{only} on $d$ so that for  $\lambda< \lambda_d$,
		\begin{equation}\label{eq: linear}
		q_{n}(\lambda) <c_d \lambda,
		\end{equation}
		for all $n>0$.
	\end{itemize}

\end{proposition}

Before proving Proposition \ref{lm: d>1 case}, we explain its content. It is immediate that $p_n(x,t)$ is increasing in $t$. Therefore, showing \eqref{eq: main} amounts to showing that the probability to connect a {time-space} point $(S_{n}^\alpha,x)$, with jobs of spatial sizes up to scale $S_n$, decays  geometrically in their spatial distance $\abs{x}%_\infty
$ under the scale $S_{n}$. This is not too surprising because $S_{n}^{\alpha}$ is the scale of seeing a job of \edtts{spatial} size $S_{n}$ \edtts{containing a particular spatial point}, and it is much larger than \edtff{$S_{n-1}^{\alpha}$, which is} the scale \edtts{for the arrival of a job of strictly smaller spatial sizes \edtts{containing} a particular site}. We expect \edtts{the speed of the spatial growths for $R_t$ to be of scale} at most $\frac{S_{n-1}}{S_{n-1}^{\alpha}} = S_{n-1}^{1-\alpha}$ when jobs of size at most $S_{n-1}$ are allowed, \edtts{while connecting two points of distance $S_n$ a spatial job of size $S_n$ gives us a speed of scale at least $S_n^{1-\alpha}$, which is much larger.} As a consequence, we expect \edtts{the speed of the spatial growths for $\mathcal{R}_t$ to be of scale $S_{n}^{1-\alpha}$ and the geometric decay in $i$} from large deviation estimates. The harder part of Proposition \ref{lm: d>1 case} is to show that $q_n(\lambda)< c_d \lambda$ uniformly in $n$, when $\lambda$ is small. Recall that when $A$ is large, we have that $S_n$ grows fast enough in $n$ \ledd{from \eqref{eq: useful ones}}: for all $n\in \mathbb{N}$, 
	\begin{equation*}  %\label{eq: growth in Sn}
	\sum_{j\leq n} 10 S_j^{1-\alpha} <S_{n+1}^{1-\alpha},  
	\end{equation*} 
	which implies that the ratio $\frac{5S_{n+1}}{8S_n} \geq 3$ when $n$ is large.
	In fact, \ledd{it} provides a recursive formula (depending on dimension $d$) for an upper bound of $q_n$ in terms of $\lambda$ and $S_n, S_{n+1}$. {In view of \eqref{eq:recursive formula} and \eqref{eq: useful ones}, we get \eqref{eq: linear} when $\lambda$ is small enough,} and we also \edt{see that condition \eqref{eq: useful ones} is \textit{not optimal}}.

For time $t> S_n^{\alpha}$, we show a corollary of Proposition \ref{lm: d>1 case}, see Corollary \ref{cor: exponential estimates} below. \eqref{eq: general} can be derived with \edtfff{a similar argument for} the second step in the proof of Proposition \ref{lm: d>1 case} below. We { give } a proof at the end of this section.
	\begin{corollary}\label{cor: exponential estimates}
	Recall that $(S_n)$ satisfy \eqref{eq: useful ones} {when $A$ is large enough}. {Let $\lambda_d$ be} from Proposition \ref{lm: d>1 case}. Then there exist \edts{positive} constants $\edts{ c'_d, \lambda_{d,0}<\lambda_d}$ depending on $d$, \edts{so that $c'_d \lambda_{d,0}<1$ and when $\lambda <\lambda_{d,0}$}
	\begin{equation} \label{eq: general}
	p_n(x,t)    \leq \left(c'_d \lambda \right)^{\left(\edts{ S_{n}^{-(1-\alpha)}\frac{\abs{x}}{t} -1}\right)  }
	\end{equation} for all positive $t>S_n^{\alpha}$.	
	\end{corollary}

To show Proposition \ref{lm: d>1 case}, we use two steps with similar ideas. The first step is an induction on $i$, which measures the spatial distance between $(0,\mathbf{0})$ and $(S_{n}^\alpha,x)$ under the scale $S_{n}$. We divide the lattice \edtt{$\mathbb{Z}^d$} into cubes with side length $\edtf{S_{n}}$. In order to connect $(0,\mathbf{0})$ and a point $(S_{n}^\alpha,x)$ with an $n$-admissible path $\gamma^{(n)}$, the $n$-admissible path { must } pass through {a sequence of neighboring cubes of length at least $i$}. Therefore, there are at most $C_d^i$ \edtt{(for example, $C_d= 3^d$)} different sequences of cubes, and for each fixed sequence, we {bound} the probability that there is an $n$-admissible path connecting them. We use a stopping time argument to reduce the problem to the case when $i=2$, see \eqref{eq: base case} below. With union bounds, we obtain that the probability decays geometrically in $i$ (see \eqref{eq: main}) from \eqref{eq: base case}. The second step is to show \eqref{eq: base case} \edtt{below}, and to obtain a recursive formula for $q_n(\lambda)$, see \eqref{eq:recursive formula}. We use induction on $n$ this time, and we { again} divide $\mathbb{Z}^d$ into cubes with side length $S_{n}$, and consider a sequence of cubes connected by an $n$-admissible paths. In both steps, the stopping times (which are the first time to connect the next cubes) are bounded by $S_n^{\alpha}$, but the total time is different: in the first step, it is $S_n^{\alpha}$; in the second step, it is $S_{n+1}^\alpha$. This results in different constants due to different counting, and different {bounds} for the probability of a fixed $n$-connecting sequence of cubes, see  \eqref{eq: reduction} and \eqref{eq: key equation} below. 

\begin{proof}(Proposition \ref{lm: d>1 case})
	Let $n$ be fixed. We can always extend an $n$-admissible path $\gamma^{(n)}$ on $[0,t]$ to an $n$-admissible path $\hat{\gamma}^{(n)}$ on $[0,s+t]$, for some $s>0$, by letting $\hat{\gamma}^{(n)}$ to be constant after time $t$,
	\[ \hat{\gamma}^{(n)}(u) = \begin{cases}\gamma^{(n)}(u)&, \text{ if $u<t$} \\
	\gamma^{(n)}(t)&, \text{ if $u\geq t$}
	\end{cases}.   \]
	Hence, as claimed above, $p_n(t,x)$ is increasing in $t$, and we only need to estimate  $p_n\left(S_{n}^{\alpha},x\right)$. We {divide  } this into two steps.
	
		\begin{enumerate}[label= Step \arabic*.]
		\item  We first assume the following statement (see \eqref{eq: base case} below) \edttfff{ but defer its proof to the second step}.  \edts{There is a $\lambda_d >0$, such that for any $\lambda<\lambda_d$, we have a constant $A_n(\lambda)<\frac{1}{4\cdot 9^d}$ depending on $\lambda$,} such that
		\begin{equation} \label{eq: base case}
			p_n(t,x)  \leq A_n\edts{(\lambda)},
		\end{equation} \edts{for all $\abs{x}%_\infty 
		\in [\frac{3}{2} S_{n}, \frac{5}{2}S_{n}]$ and $t\leq S_{n}^{\alpha}$.}
		
		Now we take a point $y \in \mathbb{Z}^d$ with $\abs{y}%_\infty
		\geq 2k S_{n}$ for some $k \in \mathbb{N}$, \edts{and we { estimate } $p_n(t,y)$, when $t\leq S_n^{\alpha}$.} By considering centers of cubes of scale $S_{n}$, we see that if $(t,y)$ and $(0,\mathbf{0})$ are $n$-connected, there exists at least $k+1$ {time-space} points $\{(t_i,x_i): i =0,1,\dots, k\}$, including  $(t_0,x_0)=(0,\mathbf{0})$ and $(t_k,x_k)=(t,y)$, such that 
		\begin{equation} \label{eq:consective cubes}
			t_{i+1} >t_i, \quad x_i \in \left(\edtts{S_{n}}\mathbb{Z}\right)^d \text{ for  $i=0,\dots, k-1$},
		\end{equation} 
		\begin{align}\label{eq: fixed points}
			\abs{x_{i+1} - x_i}%_\infty 
			= 2S_{n} \in  \left[\frac{3}{2} S_{n},\frac{5}{2} S_{n}\right]
		\end{align} 
		and 
		\begin{equation}\label{eq: connectivity induction}
			(t_i,x_i) \text{ and $(t_{i+1},x_{i+1})$ are $n$-connected.}  
		\end{equation}
		By \eqref{eq:consective cubes}, \eqref{eq: fixed points}, there are at most $\edtts{5}^{dk}$ such sequences of centers of cubes, $(x_i)_{i=1}^{k-1}$, and these sequence are deterministic. For each deterministic sequence $(x_i)^{k}_{i=0}$, we define an increasing sequence of (bounded) stopping times $u_i$ inductively, such that the increments are bounded by $S_{n}^\alpha$ and $(u_i,x_i), (u_{i+1},x_{i+1})$ are $n$-connected if $u_{i+1} -u_i <S_{n}^\alpha $:
		\begin{align}\label{eq: stopping time}
			u_0 &:=0,\\
			u_{i+1} &:= \inf\{u_{i}< t \leq u_{i}+ S_{n}^\alpha: \text{ $(u_{i},x_i)$ and  $(t,x_{i+1})$ are $n$-connected} \},
		\end{align} for $i=0,\dots, k-1$. 
		\eqref{eq: base case} indicates that each difference $u_{i+1}-u_{i} $ stochastically dominates a random variable $S_n^{\alpha}B_i$, 
		\begin{equation}\label{eq: stoch dominance, hitting time}
			u_{i+1}-u_{i} \geq S_n^{\alpha}B_i,
		\end{equation} where $B_i$ is a Bernoulli random variable with $P(B_i=1) = 1-A_n$ and $P(B_i=0) = A_n$. Then we use the strong Markov property of Poisson arrivals and Binomial random variable to get
		\begin{align} \label{eq: reduction}
			P&\left(\text{there exists} \text{ an increasing sequence $(u_i)_{i=0}^{k-1}$ such that} \right.
			\notag \\
			&\left.\text{$(u_i,x'_i)$ and $(u_{i+1},x'_{i+1})$ are $n$-connected}\right)
			\notag \\
			\leq& P\left(\sum_{i=0}^{k-1}(u_{i+1}-u_{i}) <  S_n^{\alpha}\right) \leq P\left(\sum_{i=0}^{k-1}B_i  < 1\right)
				\leq A_n^{k}.
		\end{align} Therefore, for $\abs{y}%_\infty 
	\geq 2k S_{n}$ and $t\leq S_{n}^{\alpha}$, we \ledd{bound} $p_n(t,y)$ by
		\begin{equation}\label{eq: first step final}
			p_n(t,y) \leq \edtf{3^{dk} A_n^{k}}\edts{\leq q_n^{2k} }.
		\end{equation} \edts{In particular, $q_n(\lambda) := \left(\edtts{5}^d A_n(\lambda)\right)^{\frac{1}{2}} <\frac{1}{2}$.}
		
		\item In fact, \eqref{eq: base case} is not surprising, since there is a positive probability $1-A_n$ such that no jobs of sizes $S_j$, $j\leq n$ with centers arriving in the {time-space} set $E_n:=\{(t,x): S_n<\abs{x}%_{\infty}
		< \edtff{3S_n}, t< S_{n}^{\alpha} \}$. The harder part is to show that \edts{there is a $\lambda_d>0$ \textit{independent} of $n$, such that $A_n(\lambda)$} can be $\edtf{bounded}$ uniformly in $n$ when $\lambda < \lambda_d$. We { use} induction in $n$, and obtain a recursive formula for bounds of $A_n$ in terms of $\lambda$, see \eqref{eq: recursive equation for A_n} below.
		\begin{enumerate}
			\item  The case when $n=1$ is obvious. Since the number of jobs of size $S_1$ with centers arriving in the {time-space} set $E_1$ is a Poisson random variable, \ledd{we get by \eqref{eq: arrival rate for jobs}}
			\[  P\left(\text{no job arrivals in $E_1$}\right) = \exp(-\lambda S_1^{-(d+\alpha)}|E_1|).  \] Therefore,   \eqref{eq: base case} holds for $n=1$, and \edts{there are constants $\Lambda_d$, such that if ${\lambda<\Lambda_d}$,}
			\begin{equation}\label{eq:dependency in lambda} 
				A_1(\lambda) = 1- \exp(-\lambda S_1^{-(d+\alpha)}|E_1|) \leq {6^d} \lambda, 
			\end{equation} which converges to $0$ as $\lambda$ goes to $0$.
			
			\item Assume that \edts{there is a positive $\lambda_d\leq \Lambda_d$, such that for all $n\leq k$}, \eqref{eq: base case} holds. We also assume \eqref{eq: useful ones}. We { use} an argument similar to Step 1.
			
			Let $\abs{x}%_\infty 
			\in [\frac{3}{2} S_{k+1}, \frac{5}{2}S_{k+1}]$. We get analogues of \eqref{eq:consective cubes}
			\eqref{eq: fixed points} and \eqref{eq: connectivity induction}. Notice that if $(x,t)$ and $(\edtt{\mathbf{0}},0)$ are $k+1$-connected, \edts{either} there is a job of spatial size $S_{k+1}$ with center in the {time-space} set $E_{k+1}$, or there are at least $L+1=\abs{x}(2S_{k})^{-1}+1$ many {time-space} points $(t_i,x_i)$ {in $E_{k+1}$}, $i=0,\dots, L$, that are \edtff{$k$-connected}. In particular, for  $i=0,\dots, L-1$,
			\begin{equation} \label{eq: time condi}
				t_{i+1} >t_i, \quad x_i \in \left(\edtts{S_{k}}\mathbb{Z}\right)^d
			\end{equation}
			\begin{equation} \label{eq: space condi}
				\abs{x_{i+1} - x_i}%_\infty
				 \edtts{= 2S_{k}}, 
			\end{equation}  
			and the event
			\begin{align} \label{eq: connectiv}
				G_i(s)=& \left\lbrace (t_i, x_i) \text{ and $(s,x_{i+1})$ are $k$-connected} \right\rbrace ,
			\end{align} occurs for every $s=t_{i+1}$.
			
			The probability for the first event {is bounded by}
			\begin{equation}\label{eq: ez linear bound} 
				P\left(\text{a job of size $S_{k+1}$ arrives in } E_{k+1}\right)= 1- \exp(-\lambda S_{k+1}^{-(d+\alpha)}\abs{E_{k+1}}) \leq {6^d} \lambda. 
			\end{equation}
		
			For the second event, we use arguments similar to Step 1. On one hand, from \eqref{eq: time condi} and \eqref{eq: space condi},  we have at most 
			$\edtts{5}^{dL}$  deterministic sequences $(x_j)$.
			On the other hand, for each consecutive pairs $x_i, x_{i+1}$, we consider the probability that \eqref{eq: connectiv} happens {for some $s= t_{i+1} \leq t_i+ S_k^{\alpha}$}. By the induction hypothesis, the probability is upper bounded by
			\[ P \left( G_i{(s) \ledd{\text{ occurs for some }} s\leq t_i+{S^{\alpha}_{k}} }\right) \leq A_{k}(\lambda).  \]
			Therefore, {we define a sequence of $L+1$ stopping times $(u'_i)_{i=0}^L$ recursively: $u'_0=0$, 
				\begin{equation}
					u'_i := \inf\{u'_{i-1}< s \leq S_{k+1}^\alpha:  (t_i, x_i) \text{ and $(s,x_{i+1})$ are $k$-connected }\},
				\end{equation} for $i=1,\dots,L,$ and get that each difference} $u'_i - u'_{i-1} $ stochastically dominates a random variable {$S_k^{\alpha}B'_i$}, 
			\begin{equation}\label{eq: stoch dominance2, hitting time}
			u'_i - u'_{i-1} \geq S_k^{\alpha}B'_i,
			\end{equation}
			where $B_i$ is a Bernoulli random variable with 
			\begin{equation}
				P(B'_i=0) = A_k(\lambda). 
			\end{equation}
			Hence, {for each deterministic sequence $(x_j)_{j=0}^{L+1}$, we \ledd{get} from Markov property that}
			\begin{align} \label{eq: key equation}
				&P\left(\text{there exists an increasing sequence $(t_i)_{i=0}^{L+1}$ such that $G_i(t_i)$ occurs for all $i$} \right) 
			\notag \\ 
				\leq& P\left(\edtfff{ \sum_{i=1}^{L} (u'_i-u'_{i-1})  \leq S_{k+1}^{\alpha}} \right) \leq P\left( \sum_{i=1}^{L-1} B'_i  \leq L^{\alpha} \right) \leq  \edts{\left(6 A_k(\lambda) \right) ^{\frac{L}{2}}},
			\end{align}
			where the last inequality is {from applying the Markov Inequality to a Binomial random variable $B(\edts{N},p)$}: when $a>0$, we have
			$$ 
			P(B(\edts{N},1-p)\leq \edts{K}) =P(B(\edts{N},p)\geq \edts{N-K})  \leq \mathbb{E}[e^{aB(\edts{N},p)}]e^{-a(\edts{N-K})} \leq \exp\left((e^{\edts{a}}-1)Np - a(N-K) \right).  
			$$
			When $N$ is large (such that $N^\alpha<\frac{N}{2}$), for any $K\leq N^\alpha$, we bound the last term by {taking $a = -\ln(2p)$}
			\[ \exp{ \edts{N}\left((e^{\edts{a}}-1)p-\frac{a}{2}\right)} \leq \exp \edts{N}\left(\frac{\ln(2p)}2 +(\frac{1}{2}-p) \right)  \leq \edts{ \left(6 p\right)^{\frac{\edts{N}}{2}} }. \]
			
			By \eqref{eq: ez linear bound} and \eqref{eq: key equation}, we have (from union bound) that
			\begin{equation} \label{eq: recursive equation for A_n}
				p_{k+1}(t,x) \leq {6^d}\lambda + \edtts{5^{dL}}\edts{\left(6 A_k(\lambda)\right)^{\frac{L}{2}}} =:A_{k+1}(\lambda). \end{equation}
			{Choosing \edts{$q_k(\lambda)= A_k^{\frac{1}{2}}(\lambda)$} for all \edts{$k$}, we get} \eqref{eq:recursive formula}.
			Together with \eqref{eq:dependency in lambda}, we also see from \eqref{eq: recursive equation for A_n} that there are positive constants $\lambda_d\leq \Lambda_d$ and $c_d$ depending only on $d$, such that ${A_{k+1}(\lambda) \leq c_d \lambda \edts{<1}}$, when $\lambda<\lambda_d$. Therefore, \eqref{eq: linear} follows from an induction.
		% 	 %In fact, one may find a positive fixed point for $A_k(\lambda)$ from \ref{eq: recursive equation for A_n},
		% 	 $A_{k+1}(\lambda)= A_k{\lambda} $ 
		% 	 if we assume that $\frac{S_{k+1}}{S_k}$ is a constant.
	\end{enumerate}
\end{enumerate}
\end{proof}
 
 We use a similar argument of the second step to prove \edtt{Corollary} \ref{cor: exponential estimates}. 
 
 \begin{proof}(Corollary \ref{cor: exponential estimates})
 	Without losing generality, we can assume that $x$ and $t$ are of the form ${\abs{x}%_\infty 
 		= 2k m S_n}$, and $t= m S_n^\alpha$ for some positive integers $k,m$. 
	 \edts{We first notice that the term $ S_{n}^{-(1-\alpha)}\frac{\abs{x}}{t} $ in the exponent of \eqref{eq: general}  is of order $k$ so we aim to obtain probability bounds which \ledd{are} geometric in  $k$}. Notice also that if $(t,x)$ and $(0,\mathbf{0})$ are $n$-connected, there exists at least $2m$ consecutive time-space points that are $n$-connected, \ledd{and} every two consecutive points are $kS_n$ apart. Therefore, we consider the collection of $2m$ (deterministic) points
 	\[  \mathcal{E}_m  = \{ (x_i)_{i=0}^{2m}: x_0=0,  x_i \in \left(\ledd{ S_n}\mathbb{Z}\right)^d, \abs{x_{i}- x_{i-1}}%_\infty
 	 = kS_n ,\text{ for all $i$}  \}. \]
 	 Notice that \ledd{there is a constant $C_d$ depending only on $d$} such that the cardinality of $\mathcal{E}_m$ is bounded by 
 	\begin{equation} \label{eq:  size of Em}
 		\abs{\mathcal{E}_m} \leq \left(\frac{C_d \ledd{k^{d-1}}}{4} \right)^m.
 	\end{equation}
 	 For each sequence of $(x_i)_{i=0}^{2m}  \in \mathcal{E}_m$, we define a sequence of stopping times inductively: $u_0 =0 $, and 
 	\[u_i := \inf\{s \geq u_i: (u_i,x_i) \text{ $(s,x_i)$ and  $(u_{i-1},x_{i-1})$ are $n$-connected} \}  \] By Proposition \ref{lm: d>1 case}, $u_{i+1} - u_i$ is dominated by the random variable $S_n^{\alpha}B_i$, where $B_i$ is a Bernoulli random variable with 
 	\[1-p=P(B_i=0) = q_n(\lambda)^{k-1},\]
 	and
 	\[   q_n(\lambda)< c_d \lambda<1 \] when $\lambda<\lambda_d$, for some $\lambda_d$ depending only on $d$.
 
 	Therefore, we use the union bound, the Markov property, and \eqref{eq:  size of Em} to conclude that, there exists some constants $c'_d$ and $\lambda_{d,0}<\min\{\lambda_d,\frac{1}{c_d'}\}$ depending on dimension $d$, so that when $\lambda<\lambda_{d,0}$
 	\begin{align}p_n(t,x) &\leq \sum_{ (x_i) \in \mathcal{E}_m} P\left(\exists (u_i)_{i=0}^{2m}:u_{i-1}< u_i< t,   \text{$(u_{i},x_{i})$ and $(u_{i-1},x_{i-1})$ are $n$-connected for all $i$}   \right)
 	\notag \\
 	&\leq \sum_{ (x_i) \in \mathcal{E}_m} P\left( B(2m, p) \leq m \right)  
 	\notag \\
 	&\leq  	 \abs{\mathcal{E}_m} \cdot P\left( B(2m, q_n(\lambda)^{k-1}) \geq  m \right) \leq  \left(C_d \ledd{k^{d-1}} q_n(\lambda) ^{(k-1)}\right)^m \leq (c'_d \lambda)^{k-1}, \label{eq: geo decay with integer space-time points}
 	 \end{align}
 	  where in the last line we use the classical result for a Binomial variable $B(2n,q)$ with parameters $q$ and $n$, 
 	 \[P\left(B(2n,q) \geq n\right)\leq  \sum_{k= n}^{2n} \binom{2n}{k} q^k  \leq 2^{2n} q^n.  \] Then, it is standard to derive \eqref{eq: general} from \eqref{eq: geo decay with integer space-time points} by removing the assumption that $\abs{x}%_\infty 
 	 = 2k m S_n$ and $t= m S_n^\alpha$ for positive integers $k$ and $m$.
 \end{proof}

\section{Exponential Moments of the Maximal Numbers}\label{sec: leading order}
%\edtss{TBD, an emphasis on the coefficient tailored to the }
The main object of this section is to estimate the exponential moments of the maximal numbers of jobs with the temporal size $T_j$ along $n$-admissible paths. The exponential moments help us bounding probabilities of events $A^c_{t,j}$, see \eqref{eq: overgrowth of temporal jobs}. Let
\begin{equation}\label{eq: def maximal number}
X^{(n)}_{s,t;x,j} = \sup n_j(\gamma^{(n)}_{s,t})
\end{equation} where the supremum is taken over all $n$-admissible paths $\gamma^{(n)}_{s,t}$ with the initial point $(s,x)$. Since we are mostly interested in the case when $j$ is fixed, and $s=0$, we focus on $X^{(n)}_{t} = X^{(n)}_{0,t;x,j}$. Recall that when $A$ is large enough, $(T_n ,S_n)$ satisfy properties \eqref{eq: well-separated}--\eqref{eq: useful ones}. \ledd{As in  \eqref{eq: arrival rate for jobs}, we continue to assume $c_1=1$.}  

We use an induction on $n$. There are two types of estimates for $X^{(n)}_{t} = X^{(n)}_{0,t;x,j}$, which are exponential moments tailored to different job sizes%(in parameters $a=1$ and $g_1$)%
. The first estimate, see Proposition \ref{lm: leading order}, \ledd{is the base case. It} says that the maximal number of jobs with the temporal size $T_j$ grows linearly with scale $T_j^{-(1+\delta)}$, when only jobs with spatial sizes at most $S_j$ are in the system. The second estimate, see Proposition \ref{lm: no influence} is the inductive \ledd{step. It} implies that the growth of the maximal number is not affected much by jobs with spatial sizes larger than $S_{j+1}$, when $S_{j+1}$ is large enough, see \eqref{eq: increments in b} below. We should remark that as large spatial jobs are involved, the parameter of exponential moments actually decreases. The proof of Proposition \ref{lm: leading order} { is }at the end of subsection \ref{subsec: small jumps}, and the proof of Proposition \ref{lm: no influence} { is }at the end of subsection \ref{subsec: influence from large jumps}.
%% Lemma leading order
\edts{Recall from \eqref{eq: linear} in Corollary \ref{cor: exponential estimates} that $\lambda_{d,0}$ is a constant depending only on $d$.}
\begin{proposition} (Leading Order)\label{lm: leading order}
	\rcomm{Recall that  $\delta>0$.} Assume that $T_j$ and $(S_n)_{n\leq j}$ satisfy \eqref{eq: leading order condition}, and \eqref{eq: useful ones}. Then \edts{there exists positive constants $\lambda_{d,1}<\lambda_{d,0}$, $C'_{d,1}$, depending only on $d$, such that for any $\lambda<\lambda_{d,1}$, for any $t>0$,} 
	\begin{equation}\label{eq: inductive hypo}
	\mathbb{E}\left[e^{ X^{(j)}_t} \right] \leq e^{ b t}
	\end{equation}
	where $b\edts{=C'_{d,1}} T_j^{-(1+\delta)} $.
\end{proposition}

We fix a constant \ledd{$K = \frac{d}{\alpha \kappa} > \max(\frac{d}{\alpha},3)$}, and {put $g_n  := K^{-n}$} for all $n\geq 0$. The second proposition says that we can estimate exponential moments of $X_t^{(n+1)}$ given an estimate for $X_t^{(n)}$, at a cost of reducing the exponents by a factor $K^{-1}$ and increasing the factor from $b_n$ to $b_{n+1}$ by \edtt{a small amount, see \eqref{eq: increments in b} below.}

\begin{proposition}(Influence from Large Jumps) \label{lm: no influence}
	\edtt{Recall that $g_1 = K^{-1} < \frac{\alpha}{d}$.}	Assume that for any $t>0,$ $X^{(n)}_t$ satisfies 
	\begin{equation}\label{eq: inductive hypothesis}
	\mathbb{E}\left[e^{ a X^{(n)}_t} \right] \leq e^{a b_n t}
	\end{equation} for some $b_n>0$ and some $a \in (0,1]$, then when $\lambda < \lambda_{d,0}$, there exists some positive constant $C_K$  depending on $d$, such that	\begin{equation}\label{eq: inductive step}
		\mathbb{E}\left[e^{\edtt{a}g_{\edtt{1}} X^{(n+1)}_t} \right] \leq e^{\edtt{a}g_{1} b_{n+1} t},
		\end{equation} where $b_{n+1}$ is a constant satisfying 
	\begin{equation}\label{eq: increments in b}
	b_{n+1}-b_n = C_K\cdot \edtt{a^{-1}K}\cdot S_{n+1}^{-(\alpha - \edtt{d}K^{-1})}.
	\end{equation} 
\end{proposition} 

\begin{remark}\label{rm: joint distribution} Recall that in \eqref{eq: temporal number} and \eqref{eq: def maximal number}, a job with sizes $(R,\tau)=(S_{k+j}, T_j)$ for some $k\leq n$, is counted in $n_j(\gamma^{(n)}_{s,t})$ if its arrival $(x,t')$ satisfies \eqref{eq: intersection}, \[\gamma^{(n)}_{s,t}(t'-),\gamma^{(n)}_{s,t}(t') \in B(x,S_{k+j}),   \] It is possible that the change of $\gamma^{(n)}_{s,t}$ at time $t'$, \[\abs{\gamma^{(n)}_{s,t}(t'-)-\gamma^{(n)}_{s,t}(t')}%_{\edts{\infty}}
		,\]  is small compared to  $S_{k+j}$ or there is no spatial change,\[\gamma^{(n)}_{s,t}(t'-)=\gamma^{(n)}_{s,t}(t').\] The conclusion of Proposition \ref{lm: no influence} covers these two cases.
	 \end{remark}

With an induction and \eqref{eq: concrete choice of S, T,0} for a concrete example of $(S_n,T_n)$), we get the following result from Propositions \ref{lm: leading order}, \ref{lm: no influence}. It says that a small exponential moment of $X^{(j+n)}_t$ is growing linearly in $t$ under the scale $T_j^{-(1+\delta)}$. Since the sequence $(b_{n+j})$ is increasing, we get an upper bound $b_j'$ from \edts{\eqref{eq: inductive hypo}, \eqref{eq: increments in b} and \eqref{eq: no influence from upper}}. 
\begin{corollary} \label{cor: exponential moments} Recall that $\kappa \edtff{<} \min\left(\frac{1}{8}, \frac{\delta}{4(1+\delta)} \right) $\ledd{, and $K= \frac{d} {\alpha\kappa} \in (3,\rho)$.}
	Under the assumption of Proposition \ref{lm: leading order}, \edtt{if we have $(S_{n+j})_{n\geq 1}$ with \eqref{eq: no influence from upper}
	and  \eqref{eq: useful ones},
	then} \edts{there are constants $\edttff{b'_{j}}=\edts{(C'_{d,1}+C_K)} T_j^{-(1+\delta)}$ and $\lambda_{d,1}>0$, such that} for any $t>0$ , and \edtt{$g_n = \left(\frac{\kappa}{d}\right)^{n}\cdot \alpha^{n} < \alpha^{n}$,} 
	\begin{equation} \label{eq: exp estimate}
	\mathbb{E}\left[e^{g_n X^{(j+n)}_t} \right] \leq e^{g_n \edttff{b'_{j}} t},
	\end{equation} when $\lambda<\lambda_{d,1}$.
\end{corollary}
\begin{proof} The proof is elementary. From Proposition \ref{lm: leading order} and \ref{lm: no influence}, there exists an increasing sequence $(b_{j,n})_{n\geq 1}$ such that \eqref{eq: exp estimate} holds  for any $t>0$ if we replace $b'_{j}$ by $b_{j,n}$ for any $n\geq 1$. In particular, this sequence $(b_{j,n})_{n\geq 1}$ satisfy \eqref{eq: increments in b} for all $n\geq 0$ when $b_{j,0}=b$, where $b$ is from \eqref{eq: inductive hypo}. We only need to verify that 
	\begin{equation}\label{eq: to verify}
		\edttff{b'_j} = \lim_{n\to \infty} b_{j,n} <\edts{\left(C'_{d,1}+C_K\right)} T_j^{-(1+\delta)}.\end{equation}
 We sum the differences and get
	 \[  
	 \edttff{b'_{j}}-\edttff{b_{j,0}} <  \sum_{n>j} C_K \varrho^{2(n+1)} S_{n+1}^{-\alpha(1-\kappa)}, \] \edts{which is bounded by $\edts{C_K} T_j^{-(1+\delta)}$ from \eqref{eq: no influence from upper}. Therefore, \eqref{eq: to verify} follows from  \eqref{eq: inductive hypo}.}  
\end{proof}

\subsection{Small Jumps and Leading Order}\label{subsec: small jumps}

We now prove Proposition \ref{lm: leading order} with the tools introduced in section \ref{sec: boundary process, modified paths}. The proof of Proposition \ref{lm: leading order} works for any dimension $d\geq 1$. The argument can also be extended to {dealing with $(\alpha,\beta)$ in }region V. For a fixed $j$, we drop the subscript $j$ and choose 
\edts{the increment of time $T = T_j^{1+\delta}$}. The choice of \edts{$T = T_j^{1+\delta}$} is only used in subsection \ref{subsec: small jumps}, and in subsection \ref{subsec: influence from large jumps}, we use $\edts{T} = S_{n+1}^{\alpha}$.  

Before giving the detailed proof, we summarize the steps. We {use} an elementary inequality, which can be seen as a generalization of union bounds. If we have a (deterministic) countable collection $I$ of positive random variables $Z_i$, then the exponential moment of the maximum is bounded above by the sum of the exponential moments,
\begin{equation} \label{eq: generailzed union bound}
\mathbb{E}\left[ \sup_{i\in I} e^{Z_i} \right] \leq \sum_{i\in I} E\left[e^{Z_i}\mathbb{1}_{\{Z_i\neq 0\}}\right]+\edts{P(\sup_{i \in I} Z_i=0)} \leq \sum_{i\in I} E\left[e^{Z_i}\right].
\end{equation} In our case, we want to \edts{encode} $j$-admissible paths by collections of {time-space} boxes with a number of centers, which are chosen according to some rules, see \rcomm{\eqref{eq: in cubes}, \eqref{eq: covings}, \eqref{eq: fix distance}, \eqref{eq: initialization for points}} below. \edts{The way to choose centers and boxes are similar to that in the proofs of classical Vitalli Covering Lemma or Greedy Lattice Animal argument.} Then to apply \eqref{eq: generailzed union bound}, we {take} $Z_i$ as the total number of jobs with temporal size $T_j$ \edts{inside} the collection of {time-space} boxes, and take $I$ as the set of all \edtts{sequences of} chosen centers.  On one hand, {we bound the exponential moment of $Z_i$ by using Poisson random variables and an upper bound of the probability to connect two time-space points from Corollary \ref{cor: exponential estimates}}. The {exponential moment} decays geometrically in the number of centers $m_k$ at a rate which is linear in $\lambda$, see \eqref{eq: geometric decay for prob} below. On the other hand, we count different ways of obtaining these {collections} of {time-space} boxes, which { is} exponential in the number of centers. The exponential rate for different ways only depends on dimension $d$, but not on $\lambda$. Therefore, we can estimate the exponential moment by applying \eqref{eq: generailzed union bound} when $\lambda$ is sufficiently small, see \eqref{eq: key exponential estimate} below. In particular, the argument and the bounds are independent of the index $j$.

 \begin{proof}(Proposition \ref{lm: leading order}) 
We fix time to be \edts{$t= L\cdot T$}, for some large integer $L$. Notice that \edts{$X^{(j)}_t$} is super-additive, we have \edts{$\mathbb{E}\left[e^{ X^{(j)}_{r+s}} \right] \geq   \mathbb{E}\left[e^{ X^{(j)}_r} \right] \mathbb{E}\left[e^{ X^{(j)}_s} \right]$ for any $r,s\geq 0$}. If we can show that \eqref{eq: inductive hypo} holds for a large time  \edts{$t= L \cdot T$}, then \eqref{eq: inductive hypo} also holds for any positive time. We { estimate} \edts{$\mathbb{E}\left[e^{ X^{(j)}_t} \right]$} by considering sequences of (deterministic) {time-space} boxes, which cover $j$-admissible paths.

First, we partition the {time-space} set \edts{$[0,t] \times \mathbb{Z}^d$} into boxes of the form $(s,x)+U_j$, where 
\begin{equation}\label{eq: box conditions}
U_j = [0,\edts{T}]\times[-\edts{S^{1-\alpha}_jT},\edts{S^{1-\alpha}_jT}]^d, \text{ and } (s,x) \in \left(\edts{T\cdot} \mathbb{N} \right)\times\left(2S_j^{1-\alpha} T\cdot \mathbb{Z}\right)^d  .
\end{equation} Each box has a {time-space} volume $\abs{U_j}= \edts{ (2S_j^{1-\alpha})^d \cdot T^{d+1}}$, and the number of jobs with temporal size $T_j$ and 'centers' of arrivals in the box is a Poisson random variable with mean {
$$
\lambda2^d c_2 T_j^{-(1+\beta)} \left(S_j^{1-\alpha}T \right)^d\cdot T,
$$ where $c_2$ is the normalizing constant from Theorem \ref{thm: existence model in IV}.} By condition \eqref{eq: leading order condition}, the mean is strictly smaller than
\begin{equation} \label{eq: leading order}
	\lambda  2^d {c_2}=  b_j \edts{\cdot T < 1},
\end{equation}
where $b_j = \edts{\lambda 2^d  {c_2} T_j^{-(1+\delta)}<T_j^{-(1+\delta)}}$ when $\lambda$ is strictly smaller than {$\frac{1}{2^d c_2}$}.

Then, we have encodings of paths: given a generic $j$-admissible path $\gamma^{(j)}$ on the time interval $I_k= [k\edts{T}, (k+1)\edts{T}]$, we first cover $\gamma^{(j)}$ by a collection of boxes of the form \eqref{eq: box conditions}, then we extract a sub-collection of disjoint boxes, and lastly we cover the original collection of boxes by enlarging boxes in the sub-collection. More precisely, we do the following. Since the spatial size of jobs in $j$-admissible path are at most $S_j$, we can find an integer $m_k\geq 1$ and a sequence of $m_k$ (stopping) times $s_{i,k}\in  I_k $, such that distances between two consecutive points on $\gamma^{(j)}$ at these times are between $4\edts{S^{1-\alpha}_jT}$ and $5\edts{S^{1-\alpha}_jT}$: 
\begin{align}  
s_{0,k}&= k\cdot T,\\ 
s_{i,k}&= \min\left(\edts{(k+1)\cdot T, } \inf\left\{\edts{ s_{i-1,k} \leq u \leq (k+1) T}: 4\edts{S^{1-\alpha}_jT} \leq  \abs{\edtff{\gamma^{(j)}}(u) -\edtff{\gamma^{(j)}}(\edts{s_{i-1,k}}) }%_\infty
 \leq 5\edts{S^{1-\alpha}_jT}    \right\}\right),  \label{eq: important node}
\end{align}
for $i =1,2,\dots, m_k$. For each time $s_{i,k}$, $k=1,2,\dots,m_k$, we can choose a center $x_{i,k} \in \left(2\edts{S^{1-\alpha}_jT} \cdot \mathbb{Z}\right)^d$ such that the center is $S^{1-\alpha}_j T$-close to $\gamma^{(j)}(s_{i,k})$, 
\begin{equation} \label{eq: in cubes} 
	\abs{x_{i,k} - \gamma^{(j)} (s_{i,k})}%_{\edts{\infty}} 
	\leq S^{1-\alpha}_jT. 
\end{equation} If there is more than one point in $\left(2\edts{S^{1-\alpha}_jT} \cdot \mathbb{Z}\right)^d$ satisfying 
\eqref{eq: in cubes}, we can choose according to certain rules since there are only finitely many such points.
With these centers $\mathbf{x}_k=(x_{i,k})_{\edts{i=1}}^{m_k}$, and \eqref{eq: important node},\eqref{eq: in cubes}, we obtain a covering $U_{\mathbf{x}_k}$ of the admissible path $\gamma^{(j)}$ on the time interval $I_k$ by taking unions of a collection of boxes with side lengths $10S_j$,
\begin{equation} \label{eq: covings}
U_{\mathbf{x}_k} := \bigcup_{i=1}^{m_k}  \left((k\edts{T},x_{i,k}) + [0,\edts{T}]\times[-5\edts{S^{1-\alpha}_jT},5\edts{S^{1-\alpha}_jT}]^d \right) \supset \gamma^{(j)}(I_k).  
\end{equation}
Therefore, on the time interval $I_k$, $\gamma^{(j)}$ corresponds to an encoding, which consists of $m_k$ centers $\mathbf{x}_k=(x_{i,k})_{i=1}^{m_k}$ in $\left(2\edts{S^{1-\alpha}_jT} \cdot \mathbb{Z}\right)^d$. 

For every \edts{encoding  $\mathbf{x}_k=(x_{i,k})_{\edts{i=1}}^{m_k}$ on $I_k$}, we \edts{denote by} $M_{\mathbf{x}_k}$ the number of arrivals of jobs with temporal size $T_j$ inside $U_{\mathbf{x}_k}$,
\[ M_{\edtff{\mathbf{x}}_k} = \edts{\abs{ \left\{(y,t)\in U_{\mathbf{x}_k}: (t,T_j,B)\in \Phi_y \text{ for some $B \subset [-S_j,S_j]^d$}  \right\} }}, \] and 
denote by $E_{\edtff{\mathbf{x}}_k}$ \edts{the event depending on $\mathbf{x}_k=(x_{i,k})_{\edts{i=1}}^{m_k}$,
\[ E_{\edtff{\mathbf{x}}_k} = \{ \text{the covering $U_{\mathbf{x}_k}$ contains a $j$-admissible path $\gamma^{(j)}$ with \eqref{eq: in cubes} }\}.  \]}
We consider $\mathbb{E}\left[e^{ M_{\edtff{\mathbf{x}}_k}} \cdot \mathbb{1}_{E_{\edtff{\mathbf{x}}_k}} \right]$.  By \eqref{eq: leading order}, \eqref{eq: covings}, $M_{\mathbf{x}_k}$ is a Poisson random variable with parameter at most \[ C_d m_k  b_j \edts{\cdot T}, \] for some constant $C_d$ depending on the dimension.
From \eqref{eq: important node}, \eqref{eq: in cubes}, and the fact that $\left(x_{i,k}\right)\subset \left(2\edts{S^{1-\alpha}_jT}\cdot \mathbb{Z}\right)^d$, \edts{we have that} the distance \edts{$\abs{x_{i,k} - x_{i-1,k}}%_\infty
	$}  between two consecutive centers is one of the numbers 
\begin{equation} \label{eq: fix distance}
4\edts{S^{1-\alpha}_jT} \text{ or } 6\edts{S^{1-\alpha}_jT},
\end{equation} 
\edts{for all $i<m_k$, and it can be 
\begin{equation} \label{eq: fix distance 2}
 r \cdot \edts{S^{1-\alpha}_jT},
\end{equation} where $r=0,1,2,3$, when $i=m_k$.	 }
Although these $m_k$ centers may not be distinct, there is a sequence of $\edts{m_k-1}$ points in $ I_k\times\left(2\edts{S^{1-\alpha}_jT}\cdot \mathbb{Z}\right)^d $ such that every two consecutive points are $j$-connected with the spatial distance at least $4\edts{S^{1-\alpha}_jT}$. Therefore, we get from \edts{Corollary \ref{cor: exponential estimates}} and \eqref{eq: fix distance} \edts{that } the probability $P(E_k)$ is bounded above, \edts{ for all $\lambda<\lambda_{d,0}$
\begin{equation*} 
P(E_k) \leq \edts{(c'_d \lambda )}^{6\edts{(m_k-1)}},
\end{equation*} where $c'_d,\lambda_{d,0}$ are from Corollary \ref{cor: exponential estimates}.} By the Cauchy-Schwartz Inequality, we have  
\edtt{
\begin{equation}\label{eq: geometric decay for prob}
\mathbb{E}\left[e^{ M_{\edtff{\mathbf{x}}_k}} \cdot \mathbb{1}_{E_{\edtff{\mathbf{x}}_k}} \right] \leq \exp\left(\frac{1}{2}(e^{2}-1)C_d m_k  \edts{b_j T}   \right)\edts{(c'_d \lambda )}^{\edts{3(m_k-1)}}. 
\end{equation}
}
 By choosing $x_{0,k}$ to be the last point in $\mathbf{x}_{k-1} = \left(x_{i,k-1}\right)_{i=0}^{m_{k-1}}$, for \edts{$k=1,\dots, L$}
\begin{equation} \label{eq: initialization for points}
x_{0,0} = \edtt{\mathbf{0}}, \text{ and } x_{0,k} = x_{m_{k-1}, k-1},
\end{equation} \edts{ we get from \eqref{eq: fix distance}, \eqref{eq: fix distance 2} and \eqref{eq: initialization for points} that the number of encodings $(\mathbf{x}_k)_{k=0}^{L-1}$ on the time interval $[0,t] = \bigcup_{k=0}^{L-1} I_k$ is at most 
\begin{equation} \label{eq: combinatorics}
 \prod_{k=0}^{L-1} \left( \edts{h_d}^{m_k}\right)
\end{equation} for some \edts{$h_{d}\geq 2$} depending on dimension $d$.} Now we apply \eqref{eq: geometric decay for prob} and \eqref{eq: combinatorics} to \eqref{eq: generailzed union bound}. \edts{ We} get from the independence of $\left(e^{M_{\edtff{\mathbf{x}}_k}}\cdot\mathbb{1}_{E_{\edtff{\mathbf{x}}_k}}\right)_k$, 
\begin{align}
	\mathbb{E}\left[e^{ X^{(j)}_{\edts{t}} } \right] &\leq \mathbb{E}\left[\sup_{(\edts{\mathbf{x}}_k)_{k=0}^{L-1}}\prod_{k=0}^{L-1} \left( e^{ M_{\edtff{\mathbf{x}}_k}}\cdot \mathbb{1}_{E_{\edtff{\mathbf{x}}_k}} \right) \right]  
	\notag
	\\
	&\leq \sum_{(\edts{\mathbf{x}}_k)_{k=0}^{L-1}} \prod_{k=0}^{L-1} \mathbb{E}\left[ e^{ M_{\edtff{\mathbf{x}}_k}}\cdot \mathbb{1}_{E_{\edtff{\mathbf{x}}_k}} \right] 
	\notag 
	\\
	&\leq 1+ \sum_{(m_k) \in \mathbb{Z}_+^L} \exp\left(\frac{1}{2}(e^{2}-1)C_d \edts{b_j \cdot T} \cdot \sum_{k=0}^{L-1}{m_k}    \right)\left(\edts{h_d} \edts{(c'_d \lambda )^3}\right)^{\sum_{k=0}^{L-1} \edts{(m_k-1)}}\edts{\left(h_d\right)^L}. 
\label{eq: key exponential estimate 0} 
\end{align} \edts{From \eqref{eq: leading order}, we have that $\exp\left(\frac{1}{2}(e^{2}-1)C_d b_j \cdot T \right)$ is bounded uniformly in $j$.} \edts{Therefore, from \eqref{eq: linear},  there exists some $\lambda_{d,1}\leq \lambda_{d,0}$ depending on $d$, such that 
\begin{align}\label{eq: cond 1 for lambda_d,1}
 h_d c'_d \lambda_{d,1}\exp\left(\frac{1}{2}(e^{2}-1)C_d b_j \cdot T\right) <\frac{1}{2}.
\end{align}  As a consequence, when $L$ is large, we use %\edts{the fact that $c'_d\lambda_{d,1}<\frac{1}{2}$ and 
the identity that for $q<1$, \[\sum_{(m_k) \in \mathbb{Z}_+^L} q^{\sum_{k=0}^{L-1}(m_k-1)} = \prod_{k=0}^{L-1} \left(\sum_{i=0}^{\infty}q^i \right) =   \left(\frac{1}{1-q}\right)^L  \] to bound \eqref{eq: key exponential estimate 0} by
\begin{align}\label{eq: key exponential estimate}
 \left(3 h_d \exp\left(\frac{1}{2}(e^{2}-1)C_d b_j \cdot T\right) \right)^L. 
\end{align} Notice that $b_j\cdot T =  \lambda 2^d$, we find a constant $C'_{d,1} = b \cdot T >0$ depending only on $d$, and bound \eqref{eq: key exponential estimate} by
 \[   \exp( C'_{d,1}  \cdot L   ) = \exp(b \cdot t).    \]}
\end{proof} 

\subsection{Influence from Large Jumps}	\label{subsec: influence from large jumps}
The proof of Proposition  \ref{lm: no influence} is similar to that of Proposition \ref{lm: leading order}, which also relies on \edts{Corollary \ref{cor: exponential estimates}}. We divide the proof into a few steps. \edts{Throughout this subsection, we {use } the time scale $T = S_{n+1}^{\alpha} .$} 

\edts{We first introduce a new system by adding some artificial jobs without any temporal workload at fixed {time-space} points to the original system. We consider the maximal number $Y^{(n+1)}_t$ in the new system,
	\begin{equation}\label{eq: maximal number of new system}
		 Y^{(n+1)}_{t} = \sup n_j(\gamma^{(n+1)}_{0,t}),\end{equation} where $\gamma^{(n+1)}_{0,t}$ where the supremum is taken over all $n+1$-admissible paths $\gamma^{(n+1)}_{0,t}$ with the initial point $\ledd{(0,
	\mathbf{0})}$ in the new system. \eqref{eq: maximal number of new system} is the same as 
	\eqref{eq: def maximal number} except that $Y^{(n+1)}_{t}$ is for the new system. Due to the additional artificial jobs, the maximal number $Y^{(n+1)}_t$ stochastically dominates $X^{(n+1)}_t$ in the original system, and waiting times between jobs of size strictly greater than $S_{n}$ are bounded by $S_{n+1}^{\alpha}$, which gives a natural discretization of the new system in time scale $T$.} Then, we estimate the exponential moment of $Y^{(n+1)}_t$, $\mathbb{E}\left[e^{g_{1} Y^{(n+1)}_t} \right]$ via an integral formula for a large time \edts{${t= T\cdot L}$, where $L$ is a large integer}. This formula corresponds to a decomposition of $Y^{(n+1)}_t$ according to \edts{the} occurrences of jobs of spatial sizes $S_{n+1}$ and artificial jobs. We bound terms in the formula by {\eqref{eq: general}} from Proposition \ref{lm: d>1 case}. By applying a super-additive argument to $X^{(n+1)}_t$, we get \eqref{eq: inductive step} for any small time $t$. 

More precisely, we obtain the following two lemmas. The first lemma says that we can estimate $\mathbb{E}\left[e^{g_{1} X^{(n+1)}_t} \right]$ by adding artificial jobs of radius $\frac{S_{n+1}}{4}$ at fixed {time-space} points \edts{$(s,x)$, for $s\in T\cdot \mathbb{N}$, ${x\in S_{n+1} \mathbb{Z}^d}$.} And the maximal number of the new system has an estimate in terms of a function $w_n(s,x)$. In the second lemma, we use \edts{Corollary \ref{cor: exponential estimates}} to show that $w_n(s,x)$ decays geometrically in  $x$ in the scale $S_{n+1}$.

For convenience, we say two {time-space} points $(u,x)$ and $(v,y)$ are $n$-connected via \edts{some} $S_{n+1}$ neighbors, if there exists an n-admissible path with the initial point $(u,x')$ and the final point $(v,y')$, such that 
\begin{equation}\label{eq: n-connected via  $S_{n+1}$ neighbors}
\edts{\abs{x'-x}%_\infty
	, \abs{y'-y}%_\infty
} \leq S_{n+1}.
\end{equation} We should notice that a pair of $n$-connected points are also $n$-connected via \edts{some} $S_{n+1}$ neighbors, but not vice versa. The difference is the condition on end points, see \eqref{eq: ending points distance} and \eqref{eq: n-connected via  $S_{n+1}$ neighbors}. \ledd{ Recall that $\lambda_{n+1} = \lambda S_{n+1}^{-(d+\alpha)}$.}

\begin{lemma}(A {time-space} integral)\label{lm: integral formula} Let $Y^{(n+1)}_t$ be \edts{ the maximal number} when there are additional artificial jobs with radius $\frac{S_{n+1}}{4}$, \edts{temporal size $0$,} and centers at {time-space} points \edts{$(s,x)$, where $(s,x)$ are {time-space} points in $S_{n+1}^{\alpha}\mathbb{N}\times S_{n+1}\mathbb{Z}^d $}. Then for a large time $\edts{t= T\cdot L}= S_{n+1}^\alpha L$, $L\in \mathbb{N}$, the maximal number satisfy 
	\begin{align}\label{eq: integral formula}
	\mathbb{E}\left[e^{g_{1} X^{(n+1)}_t} \right] \leq \mathbb{E}\left[e^{g_{1} Y^{(n+1)}_t} \right] 
	&\leq \sum_{N \in \mathbb{N}} \sum_{(x'_i)\in \left(\mathbb{Z}^d\right)^{N+L} } (\lambda_{n+1}t)^N e^{-\lambda_{n+1}t}\cdot
	\notag \\
	& \idotsint_{0=t_0<\dots< t_N<t} \,dt_1dt_2\cdots dt_N \prod_{i=0}^{N+L-1}  w_n(t'_{i+1}-t'_i,x'_{i+1}-x'_i),
	\end{align} where \edtt{$x_0 =\mathbf{0} \in \mathbb{Z}^d$},
	\edtt{\begin{equation}\label{eq: kernel}
	w_n(s,x) = \mathbb{E}\left[e^{g_{1} \max_{|y|\leq \frac{1}{2}S_{n+1}}  X^{(n)}_{0,s;y}} \cdot \mathbb{1}_{\{\text{$(s,x)$ and $(0,\mathbf{0})$ are $n$-connected via \edts{some} $S_{n+1}$ neighbors} \}} \right]e^{g_{1}},
	\end{equation}
	}
	and $(t'_i)_{i=0}^{N+L}$ is {the} increasing sequence which is the union of $(t_i:i\leq N)$ and multiples of \edts{$ T$} up to $t$.
\end{lemma}

\begin{proof}
	As artificial jobs of \edts{radius $\frac{S_{n+1}}{4}$} are at fixed {time-space} points, every (n+1)-admissible path in the original system is also an (n+1)-admissible path in the new system. Therefore, ${Y^{(n+1)}_t \geq X^{(n+1)}_t}$ stochastically. 
		To verify that the integral is an upper bound of $\mathbb{E}\left[e^{g_{n+1} Y^{(n+1)}_T} \right]$, we use the independence of 
		arrivals of jobs with different sizes $S_{n+1}$ and $S_k, k\leq n$.
	
	Recall that $\edts{t=S_{n+1}^\alpha L}$ is a large fixed time. 
	For every realization of arrival processes of jobs with size $S_{n+1}$, we denote by $\mathcal{C}_{n+1}$ the (random) collection of all finite sequences of consecutive arrivals of (centers of) artificial jobs or jobs with \ledd{the} spatial size $S_{n+1}$ within the time interval $[0,t]$,
	\begin{align} \label{eq: collection of jobs centers}
		\mathcal{C}_{n+1} =  \bigcup_{N=0}^\infty &\lbrace \edts{(\vec{t},\vec{x})}=(t'_i,\edts{x'_i})_{i=1}^{N+L}\subset (\mathbb{R}_+ \times \mathbb{Z}^d): t'_i<t'_{i+1}, t'_{N+L}= S_{n+1}^\alpha L, t_0 =0, x_0 =\mathbf{0},
		\notag
		\\
		&\text{ no {artificial} jobs or jobs of spatial size $S_{n+1}$ with the center $x'_{i+1}$ arrive between $(t'_i,t'_{i+1})$}  ,  
		\notag
		\\
		&\edts{\text{ a job of spatial size $S_{n+1}$ or an artificial job with the center $x'_i$ arrives at $t'_i$}}
		\rbrace.
	\end{align} Since every (n+1)-admissible path $\gamma^{(n+1)}$ can be decomposed into a collection of n-admissible paths connected by a collection of jobs with the spatial size $S_{n+1}$ and artificial jobs, each $\gamma^{(n+1)}$ corresponds to a finite sequence of (random) {time-space} points $(\vec{t},\vec{x})=(t'_i,x'_i)_{i=1}^{N+L}$ in $\mathcal{C}_{n+1}$, such that $(t'_i,x'_i)_{i=1}^{N+L}$ encodes consecutive arrivals of artificial jobs or jobs with the spatial size $S_{n+1}$ that intersect $\gamma^{(n+1)}$.
	For convenience, we take the closest center of artificial job if $\gamma^{(n+1)}$ does not intersect any artificial job at a time $iS_{n+1}^{\alpha}$, for $i=1,\dots,L$. The collection of artificial jobs \edts{``intersecting" $\gamma^{(n+1)}$} are labeled by their centers $(iS_{n+1}^{\alpha},\hat{x}_i)$, for $i=1, 2, \dots, L${; the} collection of jobs of size $S_{n+1}$ intersecting $\gamma^{(n+1)}$ are labeled by a sequence of {time-space} points $(t_i,x_i)$, $i =1, 2, \dots, N$. The union $(t_i',x_i')_{i=1}^{N+L}$ of $(iS_{n+1}^{\alpha},\hat{x}_i)_{i=1}^L$ and $(t_i,x_i)_{i=1}^N$ belongs to $\mathcal{C}_{n+1}$.

	 We obtain upper bounds of $e^{g_{1} Y^{(n+1)}_t}$ by considering different (n+1)-admissible paths characterized by elements $(t_i',x_i')_{i=1}^{N+L}$ in $\mathcal{C}_{n+1}$.   Therefore, we get an upper bound of \edtt{$e^{g_{1} Y^{(n+1)}_{\edts{t}}}$},
	\begin{align}
	       e^{g_{\edtt{1}} Y^{(n+1)}_{\edts{t}}}  \leq& \sum_{\edts{(\vec{x},\vec{t})}\in \mathcal{C}_{n+1}} \prod_{i=0}^{N+L-1}  \left(e^{g_{\edtt{1}}\cdot
	       	 \max_{|y|%_{\edts{\infty}}
	       	 	\leq S_{n+1}} X^{(n)}_{\edts{t'_i,t'_{i+1};y+x'_i}}} \right.
       \notag \\       
        &\cdot \left. \mathbb{1}_{\{\text{$(t'_i,x'_i)$ and $(t'_{i+1},x'_{i+1})$ are $n$-connected via \edts{some} $S_{n+1}$ neighbors}\}} e^{g_{\edtt{1}}}\right),
		 \label{eq: decomposition upper bound}
	\end{align} 
	where  $t_0=0, x_0=\mathbf{0}$, and the term $e^{g_{\edtt{1}}}$ is due to the fact that a job of size $S_{n+1}$ may contribute a job of temporal size $T_j$. It is not clear that \eqref{eq: decomposition upper bound} is finite on the right yet. We {take} expectation and get an upper bound which is the right hand side of \eqref{eq: integral formula}. In the proof of Proposition \ref{lm: no influence}, we {show } that the integral is finite.
	
	In fact, as job arrivals follow independent Poisson processes, we have
	\begin{equation*}
	\mathbb{E}\left[e^{g_{\edtt{1}} Y^{(n+1)}_t} \vert \mathcal{C}_{n+1} \right] \leq \sum_{(x,t)\in \mathcal{C}_{n+1}} \prod_{i=0}^{N+L-1}   w_n(t'_{i+1}-t'_i,x'_{i+1}-x'_i)
	\end{equation*} 
	where \ledd{$t_0=0, x_0=\mathbf{0}$}, and $w_n$ is defined by \eqref{eq: kernel}.  Also, for each fixed integer $N$ , \edtt{each fixed sequence of $\mathbb{Z}^d$-points $\mathbf{x} =(x'_i)_{i=0}^{N+L}$, and each fixed sequence of increasing times $(t'_i)_{i=0}^{N+L}$ ($(t'_i)_{i=0}^{N+L}$ is also the union of $(t_i)_{i= 1}^{N}$ and $\{jS_{n+1}^{\alpha}:0\leq j\leq L\}$), the density of obtaining $N$ non-artificial jobs at {time-space} points $(t'_i,x'_i)_{i=1}^{N+L}$ is}
	\[N!P(N_{\edts{t}}=N) \,dt_1dt_2\cdots dt_N =  (\lambda_{n+1}\edts{t})^N e^{-\lambda_{n+1}{\edts{t}}} \mathbb{1}_{\{0=t_0<t_1<\dots< t_N<{\edts{t}}\}} \,dt_1dt_2\cdots dt_N ,\] where $N_{\edts{t}}$ is a Poisson random variable with rate $\lambda_{n+1}{\edts{t}}$.
	As a consequence, the expectation of \eqref{eq: decomposition upper bound} is bounded by
	\begin{align}
	&\mathbb{E}\left[e^{g_{\edtt{1}} Y^{(n+1)}_{\edts{t}}} \right]  
	\notag\\ 
	\leq& \sum_{N} \sum_{(x_i)\in \left(\mathbb{Z}^d\right)^{N+L} } N! P(N_{\edts{t}}=N) \idotsint_{0=t_0<\dots< t_N<{\edts{t}}} \,dt_1dt_2\cdots dt_N \prod_{i=0}^{N+L-1}  w_n(t'_{i+1}-t'_i,x'_{i+1}-x'_i)
	\notag\\
	\leq& \sum_{N} \sum_{(x_i)\in \left(\mathbb{Z}^d\right)^{N+L} } (\lambda_{n+1}{\edts{t}})^N e^{-\lambda_{n+1}{\edts{t}}} \idotsint_{0=t_0<\dots< t_N<{\edts{t}}} \,dt_1dt_2\cdots dt_N \prod_{i=0}^{N+L-1}  w_n(t'_{i+1}-t'_i,x'_{i+1}-x'_i)
	\end{align}
\end{proof}

Due to the additional artificial jobs, we see that $t'_{i+1} - t'_i$ is at most $S_{n+1}^{\alpha}$. We estimate $w_n(s,x)$ by considering $x$ in scales of  $S_{n+1}$, and see that $w_n(s,x)$ decays geometrically in $x$ under this scale. This is the content of the second lemma, and it is a consequence of \edts{Corollary \ref{cor: exponential estimates}. }
\begin{lemma}\label{lm: regularity of kernel} \edtf{Recall that $K>3$.}
	Let $\ledd{w_n(s,x)}$ be defined by equation \eqref{eq: kernel}. Assuming $X^{(n)}_t$ satisfy \eqref{eq: inductive hypothesis}, we have the following estimates for $w_n(s,x)$, for any $\edts{s}\leq S_{n+1}^\alpha$,
	\begin{equation} \label{eq:estimate for kernel}
	\ledd{w_n(s,x)} \leq \begin{cases}
	C S_{n+1}^{\frac{\edtt{d}}{K}} e^{g_{1} (1+ b_n s)} &,   \text{ if $\abs{x}%_{\infty}
		 \leq 3 S_{n+1}$} \\
	C S_{n+1}^{\frac{\edtt{d}}{K}} r_n^{i-1}(\lambda) e^{g_{1} (1+b_n s)} &,  \text{ if $(2i-1)  S_{n+1}< \abs{x}%_{\infty} 
		\leq  (2i+1)S_{n+1}$ for some $i\geq 2$ }
	\end{cases},
	\end{equation}
	where $C$ is a constant depending on  \edtt{$d$}, and $r_n(\lambda)$ is a constant depending on $S_n$, $S_{n+1}$ and $\lambda$. Furthermore, there is a constant $C_{d,2}$ depending on $d$, such that $r_n(\lambda)<C_{d,2} \lambda $ when $(S_n)$ satisfies \eqref{eq: useful ones} and $\lambda<\lambda_{d,0}$, where $\lambda_{d,0} $ is from Corollary \ref{cor: exponential estimates}.
\end{lemma}
\begin{proof}
	We {divide} the proof into two cases. 
	\begin{enumerate}
		\item When $\abs{x}%_{\edts{\infty}} 
		\leq S_{n+1}$, assuming \eqref{eq: inductive hypothesis}, we have for any $z>0$, $\edts{s}>0$
		\begin{equation*}
		P\left( g_n (X^{(n)}_{s} - b_n \edts{s}) \geq z  \right) \leq \exp(-z)
		\end{equation*} 
		By the union bound, and stochastic dominance, we get 
		\begin{equation*}
		P\left( g_n \max_{|y|%_{\edts{\infty}}
			\leq S_{n+1}} (X^{(n)}_{0,\edts{s};y} - b_n \edts{s}) \geq d\ln (1+2S_{n+1}) + z  \right) \leq \exp(-z),
		\end{equation*} and $g_n \max_{|x|%_{\edts{\infty}}
		\leq \frac{1}{2}S_{n+1}} (X^{(n)}_{x,\edts{s}} - b_n \edts{s})$ is stochastically dominated by $\edtt{d}\ln(1+ 2S_{n+1}) + Z $, where $Z$ is an exponential random variable with rate 1. Therefore, 
		\begin{equation} \label{eq: stochastic dominance bounds}
		w_n(s,\mathbf{0}) = \mathbb{E}\left[e^{g_{\edtt{1}} \max_{|y|%_{\edts{\infty}}
				\leq S_{n+1}}  X^{(n)}_{0,\edts{s};y}} \right] e^{g_{\edtt{1}}} \leq (1+2S_{n+1})^{\frac{\edtt{d}}{K}} \frac{1}{1-\frac{1}{K}} e^{g_{\edtt{1}} (1+b_n \edts{s})}.
		\end{equation} Clearly, $w_n(s,x) \leq w_n(s,\mathbf{0})$ \edts{since the indicator function always takes value $1$ when $x=\mathbf{0}$}.
		
		\item When $(2i-1)  S_{n+1}< \abs{x}%_{\edts{\infty}}
		 \leq  (2i+1)S_{n+1}$ for some $i\geq 2$, we first get an upper bound similar to \eqref{eq: stochastic dominance bounds}. For any $s>0$,
		\[  \mathbb{E}\left[e^{ 2 g_{1} \max_{|x|%_{\edts{\infty}}
				\leq S_{n+1}}  X^{(n)}_{0,s;y}} \right] \leq (1+2S_{n+1})^{\frac{2d}{K}} \frac{1}{1-\frac{2}{K}} e^{2 g_{1} b_n \edts{s}}.  \]
	 Then by the Cauchy-Schwarz Inequality, we have 
		\begin{align} 
		\ledd{w_n(s,x)} \leq& \left(\mathbb{E}\left[e^{2 g_{1} \max_{|x|%_{\edts{\infty}}
				\leq S_{n+1}}  X^{(n)}_{s,x}} \right] P\left(\text{$(s,x)$ and $(0,\mathbf{0})$ are $n$-connected via \edts{some} $S_{n+1}$ neighbors}   \right)\right) ^{\frac{1}{2}} e^{g_{1}}  
		\notag
		\\
		\leq& (1+2S_{n+1})^{\frac{d}{K}} \left(\frac{1}{1-\frac{2}{K}}\right)^{\frac{1}{2}} e^{g_{1} (1+b_n \edts{s})} \tilde{p}_n(s,x), \label{eq: general estimate for kernel}
		\end{align}
		where $\tilde{p}_n(s,x) =  P(\text{$(s,x)$ and $(0,\mathbf{0})$ are n-connected via \edts{some} $S_{n+1}$ neighbors})^{\frac{1}{2}}$. 
		
		As $s\leq S_{n+1}^\alpha$, we get an upper bound for $\tilde{p}_n(s,x)$ by the corollary of Proposition \ref{lm: d>1 case}. More precisely, if $(s,x)$ and $(0,\mathbf{0})$ are $n$-connected via some $S_{n+1}$ neighbors, then there exists two points $y,z$ in $\left(S_n\mathbb{Z}\right)^d$ such that $\abs{y}%_\infty
		, \abs{x-z}%_\infty 
		<\frac{S_{n+1}}{2}$, and $(0,y)$, $(t,x)$ are $n$-connected. Since $(2i-1)  S_{n+1}< \abs{x}%_\infty 
		\leq  (2i+1)S_{n+1}$, we get $2(i-1)S_{n+1}< \abs{y-z}%_\infty 
		\leq  2(i+1) S_{n+1}$. Therefore, by \eqref{eq: general} in Corollary \ref{cor: exponential estimates}, we get that there exist constant $C_{d,2}>c'_d$  only depending on $d$, such that
		\begin{equation}\label{eq: random walk estimate2}
		\tilde{p}_n(s,x) \leq  \sum_{x,y} p_n(s,x-y) \leq \left(\frac{S_{n+1}}{S_n}\right)^d (c_d' \lambda )^{i\left(\frac{S_{n+1}}{S_n}\right)^{1-\alpha}} \leq r_n(\lambda)^{i-1},
		\end{equation}
		for some  $r_n(\lambda)<c_d \lambda$ when ${S^{1-\alpha}_{n+1}> 10 \sum_{j\leq n} S_j^{1-\alpha}}$, and $\lambda <\edts{\lambda_{d,0}}$. By \eqref{eq: general estimate for kernel} and \eqref{eq: random walk estimate2}, we get \eqref{eq:estimate for kernel}. 
	\end{enumerate}
\end{proof}

\begin{remark} We may get better {bounds} for $\ledd{w_n(s,x)}$. However, geometrical decay of $\ledd{w_n(s,x)}$ in $x$ is sufficient for the proof of Proposition \ref{lm: no influence}. \end{remark}

Now we prove Proposition \ref{lm: no influence} from Lemmas \ref{lm: integral formula} and \ref{lm: regularity of kernel} for a large time \edts{$t$}. And we also get the estimate for any positive time $t$ using that  $X^{(n+1)}_t$ is super-additive.

\begin{proof} (Proposition \ref{lm: no influence})
	By a change of variables $y_{i+1} = \edts{x'_{i+1}-x'_i}=(y_{i,j})_{j=1}^d \in \mathbb{Z}^d$, we rewrite \eqref{eq: integral formula} as 
	\begin{equation}\label{eq: COV}
	\sum_{N=0}^\infty \sum_{(y_i)\in \left(\mathbb{Z}^d\right)^{(N+L)} } (\edtt{\lambda_{n+1}\edts{t}})^N e^{-\edtt{\lambda_{n+1}\edts{t}}} \idotsint_{0=t_0<t_1<\dots< t_N<\edts{t}}\edts{\prod_{i=0}^{N+L-1}  w_n(t'_{i+1}-t'_i,y_{i+1}) \,dt_1dt_2\cdots dt_N} , \end{equation}
	where $(t_i')_ {i=0}^{N+L}$ is {the} increasing sequence, which is also the union of $(t_i)_{i= 1}^{N}$ and $\{jS_{n+1}^{\alpha}:0\leq j\leq L\}$.
	In view of Lemma \ref{lm: regularity of kernel}, the product
	\begin{equation}
	\prod_{i=0}^{N+L-1}  w_n(t'_{i+1}-t'_i,y_{i+1}) \leq (Ce^{g_{\edtt{1}}} S_{n+1}^{\frac{\edtt{d}}{K}})^{N+L}  e^{g_{\edtt{1}} b_n T} {\edts{r_n(\lambda)}}^{\sum \edts{y'_{i,j}}},  
	\end{equation} where 
	\edttt{$y'_i = (y'_{i,j})_{j=1}^d = \left(\max \edtff{\left\{ \lfloor \frac{\abs{y_{i,j}}-3S_{n+1}}{2S_{n+1}} \rfloor ,0\right\} } \right)_{j=1}^d. $} Therefore, for each $y'_i$, there are at most $(6S_{n+1}+1)^d$ and at least $(4S_{n+1})^d$ different $y_i$ corresponding to it. We get an upper bound for terms in \eqref{eq: COV} by summing over $y_i$ according to $y_i'$, for every fixed $N$
	\begin{align}
	%		&\sum_{(y_i)\in \mathbb{Z}^{N+L+1} } (\lambda_jt)^N e^{-\lambda_jT} \idotsint_{0=t_0<t_1<\dots< t_N<T=t_{N+1}} \,dt_1dt_2\cdots dt_N \prod_{i=0}^{N+L}  w_n(y_{i+1},t'_{i+1}-t'_i) 
	%		\notag
	%		\\ 
	%	\leq  
	&\sum_{(y'_i)\in \mathbb{N}^{N+L} } (\edtt{\lambda_{n+1}\edts{t}})^N e^{-\edtt{\lambda_{n+1}\edts{t}}} \idotsint_{0=t_0<t_1<\dots< t_N<\edts{t}} \,dt_1dt_2\cdots dt_N (6^{\edtt{d}}e^{g_{\edtt{1}}}CS^{\edtt{d}+\frac{\edtt{d}}{K}}_{n+1})^{N+L} e^{g_{\edtt{1}} b_n \edts{t}} \edts{r_n(\lambda)}^{\sum \edtt{y'_{i,j}}} 
	\notag  \\
	= &   e^{g_{\edtt{1}} b_n \edts{t}} (6^{\edtt{d}}e^{g_{\edtt{1}}}CS^{\edtt{d}+\frac{\edtt{d}}{K}}_{n+1})^{N+L} \left( \sum_{(y'_i)\in \mathbb{N}^{N+L} } \edts{r_n(\lambda)}^{\sum \edtt{y'_{i,j}}} \right) \left( (\edtt{\lambda_{n+1}\edts{t}})^N e^{-\edtt{\lambda_{n+1}\edts{t}}} \idotsint_{0=t_0<t_1<\dots< t_N<\edts{t}} \,dt_1dt_2\cdots dt_N \right),  
	\notag \\
	=& e^{g_{\edtt{1}} b_n \edts{t}} (6^{\edtt{d}}e^{g_{\edtt{1}}}CS^{\edtt{d}+\frac{\edtt{d}}{K}}_{n+1})^{N+L} \left( \frac{1}{1-\edts{r_n(\lambda)}}\right)^{\edtt{d(N+L)}}  P(N_{\edts{t}} =N) ,\label{simplification}
	\end{align}
	where $N_{\edts{t}}$ is a Poisson random variable with rate $\edtt{\lambda_{n+1}\edts{t}} = \edtff{\lambda} S_{n+1}^{-(d+\alpha)}\edts{t}$. \edts{Recall that $K>3$. By Lemma \ref{lm: regularity of kernel}, there exists a constant $\lambda_{d,2} =\min\{\frac{1}{2c_d}, \lambda_{d,0}\}$ depending only on $d$,  such that for all $\lambda< \lambda_{d,2}$, we have that} \[\edts{r_n(\lambda)}< c_d\lambda< \frac{1}{2}.\]  Therefore, we sum over $N$ for \eqref{simplification} and get an upper bound for \eqref{eq: integral formula}, 
	\begin{align}
	\mathbb{E}\left[e^{g_{1} X^{(n+1)}_{\edts{t}}} \right] \leq	&\sum_{N=0}^\infty e^{g_{1} b_n \edts{t}} (6^{\edtt{d}}e^{g_{1}}CS^{\edtt{d}+\frac{d}{K}}_{n+1})^{N+L} \left( \frac{1}{1-\edts{r_n(\lambda)}}\right)^{d(N+L)}  P(N_{\edts{t}} =N) \notag \\ 
	\leq & e^{g_{1} b_n \edts{t}}  \mathbb{E}\left[   (12^de^{g_{1}}CS^{\edtt{d}+\frac{\edtt{d}}{K}}_{n+1})^{N_{\edts{t}}}\right] (12^de^{g_{1}}CS^{\edtt{d}+\frac{\edtt{d}}{K}}_{n+1})^{L} \notag \\
	\leq & e^{g_{1} b_n \edts{t}} \exp \left((12^de^{g_{1}}CS^{\edtt{d}+\frac{d}{K}}_{n+1})\lambda_{n+1}\edts{t} + S_{n+1}^{-\alpha} \ln (12^de^{g_{1}}CS_{n+1}^{d+\frac{d}{K}}) \edts{t}  \right), \label{eq: exponential growth 0}	\end{align} \edts{where in the last line, we use the fact that
	\[\mathbb{E}\left[ e^{bN_{\edts{t}}}\right] = \exp \left((e^b-1)\lambda_{n+1}\edts{t}\right) \leq \begin{cases}
		\exp \left(3b\lambda_{n+1}\edts{t}\right), \text{ when $b\leq 1$}\\
		\exp \left(e^b\lambda_{n+1}\edts{t}\right), \text{ when $b> 1$}
	\end{cases}. 
	\] 
Since $\lambda_{n+1}S_{n+1}^{d+\frac{d}{K}} =\lambda S_{n+1}^{-(\alpha-\frac{d}{K})} < \lambda_{d,2}S_{n+1}^{-(\alpha-\frac{d}{K})} $, and $\ln(12^de^{g_{1}}CS_{n+1}^{d+\frac{d}{K}}) < C_{d,2} S_{n+1}^{\frac{d}{K}} $ for some $C_{d,2}$ depending on $d$, we bound \eqref{eq: exponential growth 0} by }
\begin{equation}	\exp \left( (g_{1}b_n + \edts{C_{d,2}}S_{n+1}^{-(\alpha-\frac{d}{K})} ) \edts{t}\right). \label{eq: exponential growth }
 \end{equation}
	
	Since $X^{(n+1)}_{\edts{t}}$ is super-additive, we have \edts{ $\mathbb{E}\left[e^{g_{1} X^{(n+1)}_{r+s}} \right] \geq   \mathbb{E}\left[e^{g_{1} X^{(n+1)}_r} \right] \mathbb{E}\left[e^{g_{1} X^{(n+1)}_s} \right]$, for any $r,s
	\geq 0$. Therefore, we also have for any $t>0$,}
\edtt{	\[
	\mathbb{E}\left[e^{g_{1} X^{(n+1)}_t} \right] \leq \exp \left( \left(g_{1}b_n + \edts{C_{d,2}}S_{n+1}^{-(\alpha-\frac{d}{K})}\right)t\right),  \] 
}
	for some $C_{d,2}$ depending on $d$ when $K>3$. Lastly, we use a change of variables to replace $g_1$ by $a\cdot g_1$.
\end{proof}

\section{Proof of the Growth Estimate}\label{sec: proof of main estimate}
In this section, we use Corollary \ref{cor: exponential moments} to prove Proposition \ref{prop: precise estimate}. By conditioning on the occurrences of jobs of spatial size $S_{i+j}$ for $j\geq 2$, we use the $g_{i+2-k}$-th exponential moments for the maximal numbers $X_{t;k}$ of jobs with temporal size $T_k$ to bound the upper tails. Applying the union bound, we {get} \eqref{eq: no large spatial job}--\eqref{eq: detail estimate}.

\begin{proof}(Proposition \ref{prop: precise estimate})  For a fixed time $t$ between $l_{2i} \cdot{\Delta t_{2i}}$ and $l_{2i+2} \cdot {\Delta t_{2i+2}}$, we first recall the event {$B_i(t)$} from \eqref{eq: overgrowth of spatial jobs}
	\begin{align} \label{eq: no S_i+j occur}
	{B_i(t)}&= \{ \text{no job of spatial size $S_{i+j}, j\geq 2,$ is along an admissible path starting from $(0,\mathbf{0})$}   \} 
	\notag \\
	&=\{ \sup m_{i+j}(\gamma_{0,t})= 0, \text{ for all $j \geq 2$}  \},
	\end{align} 
	where the supremum are taken over all admissible paths $\gamma_{0,t}$ with the initial point $(0,\mathbf{0})$. We estimate the probability $P\left({B_i(t)}^c\right)$ by looking at {time-space} boxes $E_{i+j}=  [0,t]\times[-3S_{i+j},3S_{i+j}]^d $ for all $j>1$.  If ${B_i(t)}^c$ occurs, then either there is a job of spatial size $S_{i+j}$ inside $E_{i+j}$ for some $j\geq 2$, or there is a point $x$ in $\left(2S_{i+2}\mathbb{Z}\right)^d$ with $\abs{x}%_\infty
 = 2S_{i+2}$, such that $(0,\mathbf{0})$ and $(t,x)$ are $i+1$-connected. As the arrivals for jobs of different spatial size follow independent Poisson processes, the first event occurs with a probability at most
	$$
	 \sum_{j\geq 2} 1-\exp\left(\abs{E_{i+j}}\lambda_{i+j}\right) \leq \lambda \cdot \sum_{j\geq 2}  S_{i+j}^{-\alpha} t , 
	$$
	where 
	$
	\lambda_{i+j} = \lambda P(R= S_{i+j}) \leq \lambda S_{i+j}^{-\alpha},
	$ 
	{from \eqref{eq: arrival rate for jobs}.} From Proposition \ref{lm: d>1 case} and Corollary \ref{cor: exponential moments}, the second event occurs with probability at most
	$$
	  \left(\frac{S_{i+2}}{S_{i+1}} \right)^d q_{i+1}(\lambda)^{\left(\frac{S_{i+2}}{S_{i+1}}\right)}.   
	$$
	When $\lambda<\lambda_d$, and $(S_i)$ satisfy \eqref{eq: concrete choice of S, T,0}, we have that for any $ t\in [l_{2i} \cdot{\Delta t_{2i}},l_{2(i+1)} \cdot{\Delta t_{2(i+1)}}] $ 
	\begin{equation*} \label{eq: no large jump}
		P(B^c_i(t))\leq 2C_1 S_{i+2}^{-\alpha}t
	\end{equation*} which implies  \eqref{eq: no large spatial job}.

		To get \eqref{eq: no large temporal job}, we use Corollary \ref{cor: exponential moments}  and the fact that $(X_{t;j})$ are integers. For $j>i$,
	\begin{align}
	P\left(A^c_{t,j}(0),\text{ and } {B_i(t)}\right) &=P\left(X^{(i+2)}_{t;j}\geq 1 \right)
	\notag \\
	&=P\left(X^{(j+2)}_{t;j}\geq 1 \right)
	\notag \\ 
	&\leq (\exp( 2g_{2}  T_{j}^{-(1+\delta)} t)-1)(e^{g_{2}} -1)^{-1}  
	\notag \\
	&\leq 6T_j^{-(1+\delta)} t, \label{eq: estimates1}
	\end{align} where we use $2g_2 T_{i+1}^{-(1+\delta)} t < l_{2i+2} <\frac{1}{2}$, and $\exp(x)-1\leq  3x$ for $0\leq x\leq 1$ in the last line.
	
	The arguments for \eqref{eq: usual temporal job} and \eqref{eq: unusual temporal job} are very similar. Notice that on the event ${B_i(t)}$, we have its maximal number $X_{t;j}$ the same as $X^{(i+2)}_{t;j}$  
$$ 
	X_{t;j}\mathbb = \sup n_j(\gamma_t) =X^{(i+2)}_{t;j},  
$$ 
	where the supremum is taken over all admissible paths $\gamma_t$ with the initial point $(0,\mathbf{0})$. Therefore, for any $y\geq 0$, we get
	\begin{align}\label{eq: simple union}
	P(X_{t;j} \geq y ) \leq P\left(X^{(i+2)}_{t;j} \geq y, \text{ and } {B_i(t)} \right) + P\left(B^c_i(t)\right) \leq  P\left(X^{(i+2)}_{t;j} \geq y \right) + P\left({B_i(t)}^c\right).
	\end{align}
	By Corollary \ref{cor: exponential moments}, we apply the Markov Inequality, and obtain \eqref{eq: usual temporal job},
	\begin{align*} 
	P\left(A^c_{t,j}(3), \text{ and } {B_i(t)}\right) &\leq P\left(X^{(i+2)}_{t;j} > 3 T_j^{-(1+\delta)} t \right)
	\notag \\ 
	&\leq \mathbb{E}\left[e^{g_{i+2-j} X^{(i+2)}_{t;j}} \right] \exp(-3g_{i+2-j} T_j^{-(1+\delta)} t)  
	\notag \\
	&\leq \exp( -g_{i+2-j}2T_j^{-(1+\delta)} t  ).
	\end{align*}

	Following a similar computation, we use \eqref{eq: well-separated} to get \eqref{eq: unusual temporal job} 
	\begin{align*}\label{eq: prob for moderate plate}
	P\left( A^c_{t,i}(3 \Delta t_{2i}^{2\kappa}) ,\text{ and } {B_i(t)} \right) &\leq \exp\left(g_2 r_{2i} - 3  l_{2i} \Delta t_{2i}^{2\kappa} \right) 
	\notag \\
	&\leq \exp\left( - 2\Delta t_{2i}^{\kappa}\right).
	\end{align*}
	Lastly, we use union bound and \eqref{eq: no large spatial job}--\eqref{eq: unusual temporal job} to get \eqref{eq: detail estimate}. We show only the second case \ledd{when $t$ belongs to the interval $[r_{2i} {\Delta t_{2i}}, l_{2i+2} {\Delta t_{2i+2}})$:}
	\begin{align*}
	&1- P\left( A_{t,j}(3  \Delta t_{2j}^{2\kappa} ) \text{ for all $j$}\right)\\
	 \leq& P\left({B_i(t)}^c \right) + \sum_{j> i} P\left({B_i(t)} \text{ and  $A^c_{t,j}(3  \Delta t_{2j}^{2\kappa})$ } \right) + \sum_{j\leq i} P\left({B_i(t)} \text{ and  $A^c_{t,j}(3  \Delta t_{2j}^{2\kappa})$ } \right)\\
	\leq & 2C_1 S_{i+2}^{-\alpha} t  + \sum_{j> i} 6 T_j^{-(1+\delta)}t + \sum_{j\leq i} \exp( -g_{i+2-j}T_j^{-(1+\delta)} t  ).
	\end{align*} By condition \eqref{eq: well-separated}, we get for some $C,C'>0$,	\[\sum_{j>i} 6 T_j^{-(1+\delta)} + 2C_1 S_{i+2}^{-\alpha} \leq  C'T_{i+1}^{-(1+\delta)} , \quad \sum_{j\leq i} \exp\left( -g_{i+2-j}T_j^{-(1+\delta)} t \right) \leq \exp\left( -CT_i^{-(1+\delta)} t \right),\] which implies \eqref{eq: detail estimate} for $c'=3$.
\end{proof} 

\section{Acknowledgments}
%%The authors would like to thank \edtff{Names} for valuable discussions. 
This work was supported by the \edtff{SNF grant 200021L -- 169691}.

%	
%	\textcolor{black}{
%	\begin{proposition}\label{prop: precise estimate3} \TBD 
%		Suppose $\lambda =1$, $\delta = \frac{2}{3}(\beta-1)$, $\kappa = \min\{\frac{1}{8}, \frac{\delta}{2(1+\delta)}\}$. There are three uniform positive constants $c'$, $C$ and $C'$, such that, 
%		if $ l_i \cdot{\Delta t_{2i}}<t \leq r_i \cdot {\Delta t_{2i}}$,	
%		\begin{align}\label{eq: more precise estimate4}
%		1-P\left( \left(\bigcap_{k< i} A_{t,k}(c')\right)  \textcolor{black}{\bigcap A_{t,k}(c'{\Delta t_{2i}}^{2\kappa})}\bigcap\left(\bigcap_{k>i}A_{t,k}(0) \right)\right)	<  \exp(-C \Delta t_{2i}^{\kappa}) +  C' {\Delta t_{2i+1}}^{-1} t
%		\end{align}
%		if $ r_{2i} \cdot{\Delta t_{2i}}<t \leq l_{2i+2} \cdot {\Delta t_{2i+2}}$
%		\begin{align}\label{eq: more precise estimate5} % we can simply assume that we see S_{n+1} but no S_n+2, and no T_{n+1}
%		1-P\left(\left(\bigcap_{k\leq i} A_{t,k}(c')\right) \bigcap\left(\bigcap_{k>i}A_{t,k}(0) \right)\right)	<  \exp(-C {\Delta t_{2i}}^{-1}\cdot t ) +  C' {\Delta t_{2i+2}}^{-1} t, 
%		\end{align}
%		where
%		\begin{equation*} 
%		A_{t,i}(c'):= \left\{ \sup n_i(\gamma_{0,t}) \leq c'T_i^{-(1+\delta)}t \right\},
%		\end{equation*} 
%		and the supremes are taken over the collection of admissible paths on $[0,s]$ with a common starting point $\gamma_{0,t}(0)=0$. 
%	\end{proposition} 
%	}


\begin{thebibliography}{99}\addcontentsline{toc}{chapter}{Bibliography}
	% excaple 
	%\bibitem{JBC91}J.~B.~Conway, \emph{Functions of One Complex Variable~I}. Second edition. Springer-Verlag, Graduate Texts in Mathematics~\textbf{11}, 1991.
	
	% N.T. Zinner, preprint arXiv:nucl-th/0608049. Avena's paper
	%Transient random walk in symmetric exclusion: limit theorems and an Einstein relation
	% Luca Avena, Renato dos Santos, Florian Völlering
	
	\bibitem[BCF16]{16BCF} Baccelli,~F., Chang-Lara,~H., Foss,~S.: Shape theorems for Poisson hail on a bivariate ground. Adv. in Appl. Probab. \textbf{48}, 525--543 (2016)
	
	\bibitem[BF11]{11BF} Baccelli,~F., Foss,~S.: Poisson hail on a hot ground. J. Appl. Prob. \textbf{48A}, 343--366 (2011)
	
	\bibitem[CGGK93]{93CGGK}
	Cox,~J.T., Gandolfi,~A., Griffin,~P., Kesten,~H.: Greedy lattice animals I: upper bound. Ann. Probab.
	\textbf{13}, 1151--1169 (1993)
	
	\bibitem[FKM18]{18FKM} Foss,~S., Konstantopoulos,~T., Mountford,~T.: Power law condition for stability of Poisson hail  J. Appl. Prob. \textbf{31}, 681--704 (2018)
	
	\bibitem[M02]{02M} Martin,~J.: Linear growth for greedy lattice animals. Stoch. Process. Appl. \textbf{98}, 43--66 (2002)
	
	\bibitem[W22]{22W} Wang,~Z.: Stable Systems with Power Law Conditions for Poisson Hail II. In preparation (2022)
	
	
\end{thebibliography}
\end{document}